\documentclass[11pt, a4paper]{article}

\usepackage[utf8]{inputenc}
\usepackage[T1]{fontenc}
\usepackage{lmodern}           % Better font rendering in PDF
\usepackage{microtype}         % Improves text justification and spacing
\usepackage[margin=1in]{geometry} % Standard 1-inch margins

\usepackage{amsmath, amsthm, amssymb, amsfonts, mathtools}
\DeclareMathOperator{\sgn}{sgn}
\usepackage{bbm}               % For indicator functions (e.g., \mathbbm{1})
\usepackage{bm}                % For bold math symbols

\usepackage{graphicx}
\usepackage{booktabs}          % For professional tables
\usepackage{xcolor}            % For colored text

\usepackage{authblk}

\usepackage[round]{natbib}
\usepackage[colorlinks=true, linkcolor=blue!70!black, citecolor=blue!70!black, urlcolor=blue!70!black]{hyperref}

\usepackage{url}            % simple URL typesetting
\usepackage{booktabs}       % professional-quality tables
\usepackage{amsfonts}       % blackboard math symbols
\usepackage{xcolor}         % colors
\usepackage{natbib}
\usepackage{subcaption}
\usepackage{mathtools}
\usepackage{amsthm}
\usepackage{dsfont} % AJout Abdoulaye pour \mathds
\usepackage{cleveref}
\usepackage{amssymb}
\usepackage{multirow}
\usepackage{booktabs} % For \toprule, \midrule, \bottomrule
\newcommand{\indep}{\mathrel{\mbox{$\perp\mkern-10mu\perp$}}}

\theoremstyle{plain}
\newtheorem{theorem}{Theorem}

\newtheorem{proposition}[theorem]{Proposition}

\theoremstyle{definition}
\newtheorem{definition}{Definition}

\theoremstyle{remark}

\newcommand{\samethanks}[1][\value{footnote}]{\footnotemark[#1]}

\title{OT-FairBoost: Optimal Transport-Guided Gradient Boosting \\ for Fairness Regularization on Tabular Data}

\author[1,2]{Veronika Shilova\thanks{Equall contributions.}}
\author[1,3]{Abdoulaye Sakho\samethanks}
\author[1]{ Younes Boumoussou}
\author[2]{Laurent Risser}
\author[2,4,5]{Jean-Michel Loubes}
\author[1]{Emmanuel Malherbe}

\affil[1]{Artefact Research Center, Paris, France}
\affil[2]{Institut de Mathématique de Toulouse, CNRS, Toulouse, France}
\affil[3]{Laboratoire de Probabilités, Statistique et Modélisation Sorbonne Université Université Paris Cité, CNRS, F-75005, Paris}
\affil[4]{Université de Toulouse, ANITI, Toulouse, France}
\affil[5]{INRIA Regalia, Bordeaux, France}

\date{} % Or leave empty \date{} for no date

\begin{document}

\maketitle
\renewcommand{\thefootnote}{\arabic{footnote}}
\setcounter{footnote}{0}

% --- Abstract ---
\begin{abstract}

Although neural-based machine learning models have received a lot of attention recently, tree-based models such as gradient boosting are competitive for tabular data and therefore remain widely used in various applications of AI.
As when using other machine learning predictive models, they can however yield discriminative predictions across demographic groups, due to so-called algorithmic biases.
These undesirable phenomena have motivated the emergence of new regulatory frameworks and various AI fairness strategies. 
While several pre- and post-processing methodologies exist to mitigate such bias on gradient boosting models, only a few in-processing methods have been proposed.
To bridge this gap, we introduce OT-FairBoost, a novel in-processing framework that incorporates a Wasserstein-2 distance penalty directly into the objective function of gradient-boosted trees. This OT-based mitigation strategy has been shown to efficiently optimize group fairness criteria such as Demographic Parity and Equalized Odds on neural-based predictions. To adapt this approach for gradient boosting, we extend the sample-wise gradient estimation of the Wasserstein-2 distance between group predictions to discrete distributions and hessian diagonals.
We then integrate our approach into the LightGBM training procedure and evaluate it across binary classification, regression, and multi-group sensitive attribute settings. Experimental results in each of these settings demonstrate that OT-FairBoost achieves best accuracy-fairness trade-offs against alternatives.

%Driven by new regulatory frameworks, the demand for algorithmic fairness and trustworthy AI has grown substantially. While distribution-matching regularizers based on Optimal Transport (OT) have successfully mitigated bias in differentiable models, these techniques do not transfer directly to tree ensembles, which remain the state-of-the-art for tabular data. To bridge this gap, we introduce OT-FairBoost, a novel in-processing framework that incorporates a discrete Wasserstein-2 distance penalty directly into the objective function of gradient boosted trees to optimize group fairness criteria such as Demographic Parity and Equalized Odds. To do so, based on the discrete formalism of optimal transport, we derived a sample-wise estimation of the gradient and hessian of this Wasserstein-2 distance with respect to a prediction. We integrate our approach into the LightGBM training process and evaluate it across binary classification, regression, and multi-group sensitive attribute settings. Experimental results demonstrate that OT-FairBoost achieves competitive accuracy-fairness trade-offs.
\end{abstract}

\section{Introduction}

% Fairness in AI, EU AI act
Recent years have seen a surge in demand for trustworthy AI in order 
to avoid exposing companies to scandals related to the use of their AI systems, and more generally to favor a socially non-discriminatory use of AI. A notable answer to this demand is the European Union AI Act \cite{aiact},  a regulatory frameworks which imposes to formally assess the risk level of high-risk AI systems sold from or in the EU. Algorithmic fairness is one of the key requirement of this legislation. Alongside these regulatory objectives, interest in algorithmic fairness has also grown substantially, and a large body of research has been produced in recent years \cite{mehrabi2021survey}.

%SOTA for tabular: tree-based, gradient boosting + Parler de stabular foundation models
For machine learning on tabular data, which underpins a vast majority of real-world administrative, financial, and industrial applications, Gradient-Boosted Decision Trees (GBDTs) remain a popular standard, with strong predictive performance under various datasets  \cite{sun2009classification,RF-medical-example,nguyen2021comparison,sakho2025harnessing}. Despite the recent emergence and rapid development of deep learning alternatives and large-scale tabular foundation models \cite{qu2025tabicl, hollmann2025accurate, qu2026tabiclv2}, GBDTs remain highly competitive using limited hardware resources. \cite{grinsztajn2022tree,shwartz2022tabular}. Consequently, ensuring that tree-based architectures can natively respect fairness constraints is a high-priority challenge for deployable AI.

% Fairness categories: inprocessing, postprocessing
Existing literature generally splits algorithmic bias mitigation strategies into three core paradigms: pre-processing, post-processing, and in-processing. \textit{Pre-processing methods} repair the training data itself, altering its empirical distribution at a possible cost in predictive performance. \textit{Post-processing techniques} are model-agnostic and generally computationally inexpensive. They alter outputs of a frozen model and are constrained by the representations learned during an unconstrained training phase. Notably, the in-depth study of \cite{Besse03042022} on the UCI Census dataset \cite{adultCensus96} showed that pre- and post-processing approaches can generate novel and undesirable algorithmic biases on sub-groups of the originally compared populations. This phenomenon can be limited by using \textit{in-processing methods}, which inject fairness regularizers directly into the model's optimization objective. Distribution-matching regularizers, particularly those based on Optimal Transport and Wasserstein metrics, have emerged as highly effective in-processing mechanisms for neural networks and assimilated models \cite{jiang2020wasserstein,risser2022tackling}. They however do not seamlessly transfer to tree ensembles, which is highlighted by the particularly limited amount of references on this specific subject (see Section \emph{Related Work}). %\ref{sec:related_work}). 
The discrete nature of decision tree splits, combined with the fact that samples falling into the same leaf share identical prediction values, indeed introduces severe non-differentiability that cripples standard gradient-based optimization of distribution distances.

% We propose 
%\as{We propose: Wasserstein in loss, a powerful objective to ensure fairness}
To bridge this gap, we propose OT-FairBoost, a novel in-processing framework that embeds a discrete Wasserstein-2 distance penalty directly into the objective of gradient-boosted trees. By explicitly minimizing the distance between group-conditional prediction distributions during the tree-growing phase, our approach enables native optimization of strict group fairness definitions, including Demographic Parity (DP) and Equalized Odds (EOdds), and applies to classification, regression and multi-groups  settings. 
%To overcome the non-differentiability caused by duplicated leaf-node predictions \textcolor{red}{(please explain what this is and why is that a problem)},
%\textcolor{red}{Motivate the following}
%We derive the exact samplewise directional derivatives (both gradients and Hessians) of the discrete Wasserstein distance. %Furthermore, we develop an accelerated numerical estimation using continuous cumulative distribution functions (CDFs) to ensure our approach scales efficiently to large-scale tabular datasets.
Our main contributions are:

%Contributions
%\paragraph{Contributions.} Our main contributions are:
\begin{itemize}
    \item We derive exact samplewise partial derivatives (gradients and hessians) both for the continuous and discrete Wasserstein-2 distance.
    \item We introduce OT-FairBoost, a highly scalable in-processing bias mitigation framework integrated into LightGBM that leverages cumulative distribution function (CDF) approximations for fast, samplewise derivative computations.
    \item  We demonstrate the bias mitigation abilities of OT-FairBoost in an extensive empirical validation across multiple benchmark datasets.
\end{itemize}

\section{Related work}
\label{sec:related_work}
%\paragraph{Fairness Definitions.}
Understanding how to measure fairness in machine learning systems is a crucial step to quantify and mitigate bias. Fairness is a multifaceted concept that has more than 20 corresponding metrics \cite{vzliobaite2017measuring, verma2018fairness}. It worth mentioning that these metrics are not always compatible with each other \cite{chouldechova2017fair}, which means that improving a metric score may degrade another one \cite{corbett2018measure}. In our work, we focus on \emph{group fairness} notions, which compare the behavior of the model across different populations defined by a categorical sensitive attribute. The two most popular notions are \emph{Demographic Parity} \cite{dwork2012fairness} and \emph{Equalized Odds} \cite{hardt2016equality}. Demographic Parity requires the model predictions to be independent of the sensitive attribute, whereas Equalized Odds requires their conditional independence with respect to the prediction model target.
Beyond fairness measurement, literature also focus on mitigation methods to reduce these biases in machine learning models. These strategies are commonly organized into three families according to the stage of the learning pipeline at which they intervene \cite{mehrabi2021survey}: pre-processing, post-processing and in-processing, as detailed below.

%% distinction
%\paragraph{Pre- and Post-Processing Bias Mitigation Methods.} 
%Mitigation strategies are commonly organized into three families according to the stage of the learning pipeline at which they intervene \cite{mehrabi2021survey}.

%% Pre processing
\emph{Pre-processing} methods transform the training data before learning, so that any model trained on the repaired data is (approximately) fair. Classical approaches include relabeling \cite{kamiran2009classifying}, reweighting \cite{kamiran2012data},  disparate attenuation impact by feature repair \cite{feldman2015certifying}, and  learning fair latent representations, either through discrete encodings and optimization-based transformations \cite{zemel2013learning,calmon2017optimized} or through deep generative models \cite{louizos2015variational,xu2018fairgan}. Closest to our setting, \cite{gordaliza2019obtaining} repair the input distribution by transporting the group-conditional distributions towards their Wasserstein barycenter. 
%, with guarantees on the resulting disparate impact. 
However, in general, these methods do not guarantee that an arbitrary model trained on the repaired data satisfies the target fairness criterion \cite{agarwal2022power}. %furthermore, decorrelating the sensitive attribute and target is often insufficient, since its information may remain encoded in correlated proxy features \cite{ruf2020active}.

\emph{Post-processing} methods modify the outputs of an already trained model, e.g. via group-dependent thresholds derived from ROC analysis \cite{hardt2016equality}, reject-option classification in low-confidence regions \cite{kamiran2012decision}, or optimized score transformations \cite{wei2020optimized}. More recently, promising methods map the group-conditional score distributions to their Wasserstein barycenter \cite{gouic2020projection, chzhen2020fair, gaucher2023fair}, and can also be extended to continuous sensitive attribute \cite{shilova2025fairness}. Post-processing is model-agnostic and cheap, but it acts on a frozen model and is in general suboptimal compared to methods that adapt the model itself during training.

%% Inprocessing
%\paragraph{In-Processing Bias Mitigation Methods.}
\emph{In-processing} methods incorporate the fairness requirement directly into the training objective, either as a constraint or as a regularizer, and most methods give explicit control over the accuracy--fairness trade-off. Examples include the covariance-based constraints \cite{zafar2017fairness}, empirical risk minimization under fairness constraints \cite{donini2018empirical}, the reductions approach of \cite{agarwal2018reductions}, which casts constrained fair classification as a sequence of cost-sensitive problems, and general frameworks for optimization under non-differentiable rate constraints \cite{cotter2019optimization}. Another line of work penalizes statistical dependence between the scores and the sensitive attribute, through adversarial formulations \cite{zhang2018mitigating}.
%or dependence measures such as the Hirschfeld--Gebelein--R\'enyi coefficient, which also covers continuous sensitive attributes \cite{mary2019fairness, grari2019fair}.
Distribution-matching regularizers have been proposed for differentiable models: \cite{jiang2020wasserstein} promote strong demographic parity by penalizing the Wasserstein distance between group-conditional score distributions during the training of a logistic regression, and \cite{risser2022tackling} regularize neural-network classifiers with a Wasserstein-2 penalty optimized by stochastic gradient descent. However, these approaches do not transfer directly to tree ensembles. 

% add disadvantages
For gradient boosted trees specifically, which remain a competitive and widely deployed standard for tabular data, several in-processing approaches have been proposed. FairXGBoost \cite{ravichandran2020fairxgboost} adds to the XGBoost objective a penalty to the majority group; thus, it requires this group to be identified a-priori, applies only to DP in the binary classification setting, and does not enforce equality between groups.
\cite{grari2019fair} train an adversary network against the boosted ensemble to decorrelate its scores from the sensitive attribute, so that fairness holds only to the extent that the trained adversary can detect residual dependence, without explicit control of the distributional gap.
FairGBM \cite{cruz2022fairgbm} enforces rate constraints through a proxy-Lagrangian scheme in which the non-differentiable group rates are replaced by smooth cross-entropy surrogates, so that the enforced quantity is a relaxation of the target rates rather than the fairness criterion itself. Notably, none of these approaches directly minimizes a distributional distance between group-conditional prediction distributions, and their formulations are geared towards binary classification.

% \citep{jiang2020wasserstein} introduce a procedure to promotes group fairness by constraining the distance between the distributions of model scores (or predictions) across sensitive groups. By expressing fairness as a Wasserstein-distance regularizer or constraint, their approach enables an explicit trade-off between accuracy and fairness during the training of a logistic regression.

\section{Model}
\label{model}

%\subsection{Problem and Notations}
%%Notations
\paragraph{Notations.}
Assume we are given a training dataset $\mathcal{D} = \{(x_i, y_i, s_i)\}_{i=1}^n$ consisting of $n$ independent and identically distributed (i.i.d.) samples drawn from a joint distribution, formally $(x_i, y_i, s_i) \sim (X,Y,S)$ where $X$, $Y$ and $S$ are random variables. Here, the support of $X$ is $\mathbb{R}^d$, the feature space, so that $x_i$ is the features vector of $i$-th sample. The support of $S$ is $\{0, 1\}$, $s_i$ representing the binary sensitive demographic attribute of $i$-th sample, that can be easily extended to multi-groups (as in our experiments).
Last, $Y$ represents the target variable, whose support depends on the task context. For binary classification problems, this support is $\{0,1\}$, for regression problems, it is $\mathbb{R}$. 

%Indeed, for classification, we aim to output probabilities by applying the sigmoid function to the raw margins (logits). To compute gradients with respect to these raw margins as required by gradient boosting, we apply the chain rule. 
%Defining $\hat{p}_i = \sigma(\hat{z}_i) = \frac{1}{1 + \exp(-\hat{z}_i)}$, the gradient of the loss function $\ell$ with respect to the raw margin is given by: $\frac{\partial \ell(y_i, \hat{z}_i^{(t)})}{\partial {\hat{z}_i^{(t)}}} \left(\hat{z}_i^{(t-1)}\right) = \frac{\partial \ell(y_i, \hat{p}_i^{(t)})}{\partial {\hat{p}_i^{(t)}}} \left(\hat{p}_i^{(t-1)}\right) \times \frac{\partial \hat{p}_i^{(t)}}{\partial {\hat{z}_i^{(t)}}} \left(\hat{z}_i^{(t-1)}\right)$ \textcolor{red}{TBC}

\begin{definition}\label{def:dp_clf}
    A binary classifier $\hat{Y}$ satisfies Demographic Parity (DP) if $\hat{Y} \indep S$, i.e. $\mathbb{P}(\hat{Y} = 1 | S = 0) = \mathbb{P}(\hat{Y} = 1 | S = 1)$.
\end{definition}

\begin{definition}\label{def:eodds_clf}
    A binary classifier $\hat{Y}$ satisfies Equalized Odds (EOdds) if $\hat{Y} \indep S \mid Y$, i.e. $\mathbb{P}(\hat{Y} = 1 | S = 0, Y = y) = \mathbb{P}(\hat{Y} = 1 | S = 1, Y = y), \forall y \in \{0, 1\}$.
\end{definition}

\begin{definition}\label{def:dp_reg}
    A regressor $\hat{Y}$ satisfies Demographic Parity (DP) if $\hat{Y} \indep S$, i.e. $\sup_{\tau \in \mathbb{R}}|\mathbb{P}(\hat{Y} \leq \tau | S = 0) - \mathbb{P}(\hat{Y} \leq \tau | S = 1)| = 0$.
\end{definition}

%\begin{definition}\label{def:eodds_reg}
%    A regressor $\hat{Y}$ satisfies Equalized Odds (EOdds) if $\hat{Y} \indep S \mid Y$, i.e. $\sup_{\tau \in \mathbb{R}}|\mathbb{P}(\hat{Y} \leq \tau | S = 0, Y \in dy) - \mathbb{P}(\hat{Y} \leq \tau | S = 1, Y \in dy)| = 0, \forall y \in \mathbb{R}$, where $dy$ is an infinitesimal interval around y.
%\end{definition}

\subsection{Boosting procedure and fair loss}

We build upon the functional gradient boosting paradigm exemplified by frameworks like XGBoost \citep[][]{chen2016xgboost} and LightGBM \citep[][]{ke2017lightgbm}. The model learns to estimate $\hat y$ from sum of weak learners $F_t=\sum_{k=1}^t f_k(x)$.
In the case of classification, this continuous output is converted to a probability using a sigmoid function $\sigma(F_t(x))=\frac{1}{1+e^{-F_t(x)}}$, and represents the estimated $P(\hat y=1)$.
%
%\begin{align*}
%\hat y  = \hat z & \text{ if } y \in \mathbb{R} \\
%\mathbb{P}(\hat y=1)  = \sigma(\hat z) & \text{ if } y \in \{0,1\} 
%\end{align*}
%where $\sigma(z)=\frac{1}{1+e^{-z}}$ is the sigmoid function and approximates the probability of $\hat y=1$ for binary classification.

%$\hat z=F_t$ for regression, and $\hat z=\sigma\left(F_t\right)$ with $\sigma$ the sigmoid function for classification. 
To iteratively learn $F_t$, at each boosting iteration $t$, a new decision tree $f_t$ is added to the model ensemble in order to improve the predictions on the dataset $\mathcal{D}$, as quantified by a samplewise loss $\ell\big(y_i, F_t(x_i)\big)$ (e.g., binary cross-entropy or squared error). %For each sample $x_i$, the cumulated prediction becomes $F_t(x_i) = F_{t-1}(x_i) + f_t(x_i)$, where the tree $f_t$ aims at minimizing a second-order Taylor expansion of the objective function around the previous iteration's predictions $F_{t-1}(x_i)$:
The tree $f_t$ aims at minimizing a second-order Taylor expansion of the objective function around the previous iteration's predictions $F_{t-1}(x_i)$:
\begin{align}
\mathcal{L}^{(t)}  &= \sum_{i=1}^n \ell\!\left(y_i, F_{t-1}(x_i) + f_t(x_i)\right) + \Omega(f_t) \\
&\approx \sum_{i=1}^n \left[ \ell\!\left(y_i, F_{t-1}(x_i)\right) + g_i f_t(x_i) + \frac{1}{2} h_i f_t^2(x_i) \right] + \Omega(f_t) \nonumber \\
& = \tilde{\mathcal{L}}^{(t)},
\end{align}
where $\ell(\cdot, \cdot)$ is twice differentiable, with first-order gradients $g_i$ and second-order Hessians $h_i$ defined samplewise as:
\begin{align}
g_i &= \frac{\partial \ell\left(y_i, {F}_{t}(x_i)\right)}{\partial {{F}_{t}(x_i)}} \left({F}_{t-1}(x_i)\right), \\ 
h_i &= \frac{\partial^2 \ell\left(y_i, {F}_{t}(x_i)\right)}{\partial {\left({F}_{t}(x_i)\right)^2}} \left({F}_{t-1}(x_i)\right).
\end{align}
%\begin{equation}
%g_i = \frac{\partial \ell(y_i, \hat{z}_i^{(t)})}{\partial {\hat{z}_i^{(t)}}} \left(\hat{z}_i^{(t-1)}\right),  \; \; h_i = \frac{\partial^2 \ell(y_i, \hat{z}_i^{(t)})}{\partial {\left(\hat{z}_i^{(t)}\right)^2}} \left(\hat{z}_i^{(t-1)}\right).
%\end{equation}
% \begin{equation}
% g_i = \frac{\partial \ell(y_i, \hat{y}_i^{(t-1)})}{\partial \hat{y}_i^{(t-1)}} \quad \text{and} \quad h_i = \frac{\partial^2 \ell(y_i, \hat{y}_i^{(t-1)})}{\partial {\left(\hat{y}_i^{(t-1)}\right)^2}}.
% \end{equation}
%
and the term $\Omega(f_t) = \rho T + \frac{1}{2}\eta \sum_{j=1}^T w_j^2$ regularizes the structural complexity of the tree $f_t$.

The practical algorithm to minimize this $\tilde{\mathcal{L}}^{(t)}$ and the corresponding leaf split criterion is detailed in \cite{chen2016xgboost}. Besides, please note that in the classification setting, $\ell$ is implicitly composed with $\sigma$ to obtain probabilities in $[0,1]$ from $F_t(x_i)$. In this case, the derivative includes a factor $\sigma'(F_t(x))$ coming from the chain rule.
%

%The term $\Omega(f_t) = \gamma T + \frac{1}{2}\eta \sum_{j=1}^T w_j^2$ regularizes the structural complexity of tree $f_t$ by penalizing its total number of leaves $T$ and the $L_2$ norm of its leaf weights $w \in \mathbb{R}^T$. By setting the derivative of the localized objective with respect to the weights to zero, the optimal weight $w_j^*$ for a specific leaf node $j$ governing the sample index set $I_j$ is analytically given by:
%\begin{equation}
%w_j^* = -\frac{\sum_{i \in I_j} g_i}{\sum_{i \in I_j} h_i + \eta}.
%\end{equation}

To enforce algorithmic fairness during this iterative optimization, we must formalize the statistical dependence between the model's predictions and the sensitive attribute. 
% Standard group fairness criteria are fundamentally characterized by disparities between group-conditional prediction distributions. For instance, Demographic Parity (DP) mandates that the marginal distribution of predictions be independent of the sensitive attribute ($\hat{Y} \perp S$), whereas Equalized Odds (EOdds) requires conditional independence given the ground truth label ($\hat{Y} \perp S \mid Y$).
To operationalize \Cref{def:dp_clf} and \Cref{def:dp_reg} over continuous prediction scales, let us focus on the Demographic Parity. To ease notations, we will write $\hat{z} \in \mathbb{R}$ the continuous prediction of the model, in the case of classification and regression: 
\begin{align}
\hat z^{(t)} = 
\begin{cases} 
   F_t(x) & \text{if } y \in \mathbb{R} \\
  \sigma\left(F_t(x)\right)   & \text{if } y \in \{0,1\} 
\end{cases}
\end{align}
We group $\hat z_i^{(t)}$ values by sensitive attribute $s_i$. We define the empirical set of prediction scores at iteration $t$ for group $s$ as
$$\mathcal{Z}^s = \{\hat{z}_i^{(t)} \in \mathbb{R} \mid s_i = s\}=\{{\hat{z}_j}^{s^{(t)}} \in \mathbb{R} \mid j = 1, \dots, n_s\}$$
where $n^s = |\{i : s_i = s\}|$ denote the size of sensitive group $s \in \{0, 1\}$, and each group is re-indexed. In the following, we omit to write $^{(t)}$ for clarity. The empirical probability measure $\mu^{\mathcal{Z}^s}$ for the demographic cohort $\mathcal{Z}^s$ is expressed as:
\begin{equation}
\mu^{\mathcal{Z}^s} = \frac{1}{n_s} \sum_{j=1}^{n^s} \delta_{\hat{z}^s_j},
\end{equation}
where $\delta_z$ denotes the Dirac delta measure centered at $z$. 

We propose an in-processing regularization framework that directly incorporates a statistical distance metric $\mathcal{M}$ between these group-conditional empirical measures $\mu^{\mathcal{Z}^0}$ and $\mu^{\mathcal{Z}^1}$ into the boosting objective at iteration $t$:
\begin{equation}
\mathcal{L}^{(t)}_{\text{fair}} = \mathcal{L}^{(t)} + \lambda \mathcal{M}\!\left(\mu^{\mathcal{Z}^0}, \mu^{\mathcal{Z}^1}\right),
\end{equation}
where $\lambda \in \mathbb{R}^+$ is a user-defined hyperparameter controlling the accuracy-fairness trade-off frontier. The functional structure of $\mathcal{M}$ can be adapted to accommodate diverse fairness definitions (e.g., conditioning on ground truth $y$ to satisfy EOdds). In this work, we consider the discrete Wasserstein distance, to follow the nature of our distributions $\mu^{\mathcal{Z}^s}$. 

\subsection{Discrete Wasserstein distance}

In the case of demographic parity,
we consider the Wasserstein-2 distance between the two distributions of predicted values
\begin{equation}
    \label{eq:w_2_discrete}
\mathcal{W}_2^2\big(\mu^{\mathcal{Z}^0}, \mu^{\mathcal{Z}^1}\big) = \min_\gamma \sum_{i=1}^{n^0} \sum_{j=1}^{n^1} \gamma_{i,j} ( \hat z^0_i - \hat z^1_j )^2
\end{equation}
where $\gamma \in \mathbb{R+}^{n^0 \times n^1}$ is the coupling matrix that verifies $\sum_j \gamma_{i,j} = \frac{1}{n^0} \; \forall i$ and $\sum_i \gamma_{i,j} = \frac{1}{n^1} \; \forall j$. Note that when there exists duplicated values in $\mathcal{Z}^0$ or $\mathcal{Z}^1$, while the minimum exists the corresponding solution $\gamma$ is not unique. In the case of distinct values, the solution $\gamma^*$ is unique, and is invariant as long as the ranking among values within each group and respective probability masses are preserved \cite{peyre2019computational}. 
However, in our case of $\hat z_i$ predicted by a tree-based model, all samples falling in the same leaf will have the same $\hat z_i$ values, and even after few iterations and aggregated decision trees, having unique $\hat z_i$ values is unlikely.

To simplify the calculations below, we will define the set of distinct values of $\mathcal{Z}^s$:
$\mathcal{U}^s = \{u_1^s, \dots, u_{{|\mathcal{U}^s|}}^s\}$, where ${|\mathcal{U}^s|} \leq n^s$. We define the mapping to the new index between $\hat z^s_i$ and the corresponding $u^s_j$ value, $d: i \to d(i)=j$ such that $u^s_j = \hat z^s_i$, where we omit to clarify the group $s$ since it will be obvious. We define the probability mass of a value $u$ by
$
p^s(u) = \frac{|\{\hat z\in \mathcal{Z}^s| \hat z=u
\}|}{n^s}
$
The distribution $\mu^s(z)$ can thus be rewritten as
$
\mu^{\mathcal{Z}^s}(z) = \sum_{j=1}^{|\mathcal{U}^s|} p^s(u^s_j) \delta_{u^s_j}(z)
$
and the Wasserstein distance as
\begin{equation}
\label{eq:w2_unique}
\mathcal{W}_2^2\big(\mu^{\mathcal{Z}^0}, \mu^{\mathcal{Z}^1}\big) = \sum_{i=1}^{|\mathcal{U}^0|} \sum_{j=1}^{|\mathcal{U}^1|} \gamma^*_{i,j} ( u^0_i - u^1_j )^2
\end{equation}
where the coupling matrix $\gamma^* \in \mathbb{R+}^{|\mathcal{U}^0| \times |\mathcal{U}^1|}$ is the unique solution of minimization that verifies $\sum_j \gamma^*_{i,j} = p^0(u^0_i) \; \forall i$ and $\sum_i \gamma_{i,j} = p^1(u^1_j) \; \forall j$. In the following, we will only consider this $\gamma^*$ as defined for the sets of unique values $\mathcal{U}^0$ and $\mathcal{U}^1$.
Let $T_{\mathcal{U}^0 \to \mathcal{U}^1}$ be the corresponding transport map from $\mathcal{U}^0$ to $\mathcal{U}^1$:
\begin{equation}\label{eq:mapping}
    T_{\mathcal{U}^0 \to \mathcal{U}^1} (u^0_i) =  \sum_{j=1}^{|\mathcal{U}^1|} \frac{\gamma^*_{i,j}}{p^0(u^0_i)} u^1_j
\end{equation}
%where $ \sum_{j=1}^{|\mathcal{U}^1|}   \frac{\gamma^*_{i,j}}{p^0(u^0_i)}=1$, so that $T_{\mathcal{U}^0 \to \mathcal{U}^1} (u^0_i)$ is
that can be interpreted as the barycenter of  points $u^1_j \in \mathcal{U}^1$ corresponding to $u_i^0$ given by the optimal coupling matrix $\gamma^*$. 

We extend this definition to the original values sets $\mathcal{Z}^0$ and $\mathcal{Z}^1$, that we write $T_{\mathcal{Z}^0 \to \mathcal{Z}^1}$. If $\hat z_i^0$ is unique in $\mathcal{Z}^0$, then the transport map is directly $T_{\mathcal{Z}^0 \to \mathcal{Z}^1}(\hat z^0_i) = T_{\mathcal{U}^0 \to \mathcal{U}^1} (u^0_{d(i)})$. When $\hat z_i^0$ is not unique, we define it as
$$
T_{\mathcal{Z}^0 \to \mathcal{Z}^1}(\hat z^0_i) = \lim_{\substack{\varepsilon \to 0 \\ \varepsilon > 0}} \; T_{\mathcal{Z}^0_\varepsilon \to \mathcal{Z}^1} (\hat z^0_i + \varepsilon)
$$
where $\mathcal{Z}^0_\varepsilon$ is the set of value where $\hat z^0_i$ has been changed to $z^0_i + \varepsilon$, and is thus unique for sufficiently small $\varepsilon$, i.e.  $\varepsilon < \min_{i \neq j} |\hat z^0_i - \hat z^0_j|$. This limit exists, the optimal coupling matrix $\gamma^*$ is indeed constant since the order among unique values of $\mathcal{Z}^0_\varepsilon$ is kept as well as their probability masses. Note that for $\varepsilon < 0$ the limit would differ, since the order of unique values in $\mathcal{Z}^0_\varepsilon$ would be different and thus would be $\gamma^*$.
The transport maps $T_{\mathcal{U}^1 \to \mathcal{U}^0}$ and $T_{\mathcal{Z}^1 \to \mathcal{Z}^0}$ are equivalently defined.

\subsection{Samplewise Derivatives}

\begin{proposition}
\label{proposition:main-derivatives}
We use as local estimate of the gradient and hessian of the Wasserstein-2 term:
\begin{align}
\label{eq:grad}
\frac{\partial \mathcal{W}_2^2\big(\mu^{\mathcal{Z}^0}, \mu^{\mathcal{Z}^1}\big)}{\partial {\hat z^0_i}} \left({\hat{z}_i}^{0(t-1)}\right) &= \frac{2}{n^0} \left(\hat z^0_i - T_{\mathcal{Z}^0 \to \mathcal{Z}^1}(\hat z^0_i)\right) \nonumber\\
\frac{\partial^2 \mathcal{W}_2^2\big(\mu^{\mathcal{Z}^0}, \mu^{\mathcal{Z}^1}\big)}{\partial \left({\hat z^0_i}\right)^2} \left({\hat z_i}^{0(t-1)}\right)&= \frac{2}{n^0}
\end{align}
%\begin{align}
%\label{eq:grad}
%\frac{\partial^+ \mathcal{W}_2^2(\mu^{\mathcal{Y}^0}, \mu^{\mathcal{Y}^1})}{\partial {y^0_i}} (y^0_i) &= \frac{2}{n^0} \big(y^0_i - T_{\mathcal{Y}^0 \to \mathcal{Y}^1}(y^0_i)\big) \nonumber\\
%\frac{\partial^{2} \mathcal{W}_2^2(\mu^{\mathcal{Y}^0}, \mu^{\mathcal{Y}^1})}{\partial {y^0_i} ^2} (y^0_i)&= \frac{2}{n^0}
%\end{align}
where the partial derivatives are implicitly evaluated for all $\mathcal{Z}^s$ values at ${\hat z_i}^{s(t-1)}$. Our estimate $\frac{\partial \mathcal{W}_2^2\big(\mu^{\mathcal{Z}^0}, \mu^{\mathcal{Z}^1}\big)}{\partial {\hat z^1_j}}$ for a sample in the other group is similarly given with the transport map $T_{\mathcal{Z}^1 \to \mathcal{Z}^0}$, equivalently defined.
\end{proposition}

The proof of \Cref{proposition:main-derivatives} is provided in the Appendix. This proposition establishes our estimate for the derivatives of the Wasserstein distance, relying primarily on its discrete definition. Please note that in this formalism, $\mathcal{W}_2^2(\mu^{\mathcal{Z}^0}, \mu^{\mathcal{Z}^1})$ is not $\mathcal{C}^1$ in points $\hat z_i^0$ with  duplicated values, so that we use a right-hand first-order derivative ($\varepsilon \in \mathbb{R}^+ \to 0$). The left-hand derivative has a similar expression, that tends to the same value for a sufficiently large dataset $\mathcal{D}$ (see Appendix for more details).
Furthermore, we also derived the partial derivatives in the continuous formulation of the samples distribution, leveraging their respective cumulative distribution function and the continuous Wasserstein-2 distance in 1D. To do so, we defined a perturbation around the sample $\hat z_i^0$ as $\alpha_{\varepsilon}^{i} = \Delta (-\mathds{1}_{[\hat z_i^0 - \varepsilon, \hat z_i^0[} +\mathds{1}_{[\hat z_i^0, \hat z_i^0+\varepsilon]})$. The results are aligned with our expressions of Proposition 1, up to the $\Delta$ value, with no difficulty for duplicated sample values thanks to continuous formalism. The details of this derivation are in the Appendix.

\begin{table*}[t]
\centering
\caption{Summary of experimental datasets and corresponding fairness problems.}
\small
\setlength{\tabcolsep}{4pt} 
\label{tab:datasets_summary}
\small
\begin{tabular}{lcccc}
\toprule
Dataset & Samples ($n$) & Sensitive Attribute $S$ & Task Type & Target Variable $Y$ \\ \midrule
Folktables TN & 34003 & Sex (binary)  & Binary Class. & Income $> \$50$K \\
 &  &  \& Ethnicity (Multi-group) &  &  \\
FairJob & $>$ 1,000,000 & Gender (Proxy) & Binary Class. & Ad Interaction \\
Communities \& Crime & 1994 & $\mathds{1}_{\text{Black Pop. Ratio}\ge 0.23}$& Regression & Crime Rate \\ 
\bottomrule
\end{tabular}
\end{table*}
To incorporate the samplewise derivatives into the objective of the gradient boosting framework, we assume that for a given sample $\hat z^s_i$, the other group distribution $\mu^{\mathcal{Z}^{1-s}}$ is fixed, which amounts to say that in a boosting iteration we neglect the variation in the other group. Consequently, $
\frac{\partial^{2} \mathcal{W}_2^2(\mu^{\mathcal{Z}^0}, \mu^{\mathcal{Z}^1})}{\partial {\hat z^0_i} \partial {\hat z^1_k}}=0$ for all $i=1..n^0,k=1..n^1$.
We use as local estimate of the other non-diagonal hessian terms:
$$
\frac{\partial^{2} \mathcal{W}_2^2(\mu^{\mathcal{Z}^0}, \mu^{\mathcal{Z}^1})}{\partial {\hat z^0_i} \partial {\hat z^0_j}}=0 ,  \forall i \neq j
$$
and equivalently for samples in the other group $\mathcal{Z}^1$.
This expression is straightforward in the case of distinct ${\hat z^0_i}$ and ${\hat z^0_j}$ values, following \Cref{eq:w_2_discrete}. When ${\hat z^0_i}={\hat z^0_j}$, then these partial derivative are implicitly one-sided (right and left-handed), since $\mathcal{W}_2^2(\mu^{\mathcal{Z}^0}, \mu^{\mathcal{Z}^1})$ is not $\mathcal{C}^1$ in this point. We detail in the Appendix how we obtain this expression, with numerical illustrations.

%As an assumption, we consider that all non-diagonal terms of the Hessian matrix are null,
%$\frac{\partial^{2} \mathcal{W}_2^2(\mu^{\mathcal{Z}^0}, \mu^{\mathcal{Z}^1})}{\partial {\hat z^0_i} \partial {\hat z^0_j}}=0$ for $i \neq j$. It is true when $\hat z^0_i \neq \hat z^0_j$, since a perturbation in $\hat z^0_j$ will have no impact on the mapping value $T_{\mathcal{Z}^0 \to \mathcal{Z}^1}(\hat z^0_i)$ of sample prediction $\hat z^0_i$ -- equivalently, there is no term crossing $\hat z^0_i$ and $\hat z^0_j$ in Equation \ref{eq:w_2_discrete}. We consider it to be also valid for duplicated sample predictions, since it is becomes less and less frequent during boosting iterations, in order to simplify the loss function to sum of sample-wise terms. Last, for two samples $\hat z^0_i, \hat{z^1_j}$respectively in group $s=0$ and $s=1$, $\frac{\partial^{2} \mathcal{W}_2^2(\mu^{\mathcal{Z}^0}, \mu^{\mathcal{Z}^1})}{\partial {\hat z^0_i} \partial {\hat z^1_j}}=0$ when $\gamma^*_{d(i),d(j)}=0$ (see Equation \ref{eq:w_2_discrete}). As an approximation, we consider it to be true for all pairs of samples in the two groups, since the non null terms $\gamma^*_{d(i),d(j)}$ tends to zero when the dataset size grows. Consequently, in practice, after a Taylor expansion of $\mathcal{W}_2^2(\mu^{\mathcal{Z}^0}, \mu^{\mathcal{Z}^1})$ the objective is to minimize the following expression

In practice, following our estimate of non-diagonal hessian terms, after a Taylor expansion of $\mathcal{W}_2^2(\mu^{\mathcal{Z}^0}, \mu^{\mathcal{Z}^1})$ the objective is to minimize
\begin{align}
\tilde{\mathcal{L}}^{(t)}_{\text{fair}}  = &\tilde{\mathcal{L}}^{(t)} + \mathcal{W}_2^2\big(\mu^{\mathcal{Z}^0}, \mu^{\mathcal{Z}^1}\big)
\nonumber \\
& + \sum_{i=1}^n \left[ \frac{\partial \mathcal{W}_2^2\big(\mu^{\mathcal{Z}^0}, \mu^{\mathcal{Z}^1}\big)}{\partial F_t(x_i)} \big(F_{t-1}\left(x_i\right)\big)
 f_t(x_i)   \right. \left.  + \frac{1}{2} \frac{\partial^2 \mathcal{W}_2^2\big(\mu^{\mathcal{Z}^0}, \mu^{\mathcal{Z}^1}\big)}{\partial ({F_t(x_i))^2}} \big(F_{t-1}\left(x_i\right)\big) f_t^2(x_i) \right]
\end{align}
where the partial derivatives are given by \Cref{eq:grad}, with the group value $s_i$ indicating the transport map direction. Please note that for classification, one need to multiply the gradient and hessian of \Cref{eq:grad} by $\sigma'(F_t(x_i))$.

%In the following, we propose an efficient numerical estimation and evaluate the quality of these estimates.

%\subsubsection{Relation between DP and EOdds}

%\subsubsection{Wasserstein distance with DP and EOdds}

\subsection{Numerical Estimation of the Derivatives}

In practice, we compute the gradient of \Cref{eq:grad} by using the CDFs for $\mu^{\mathcal{Z}^0}$ and $\mu^{\mathcal{Z}^1}$ for the transport map term:
$$
T_{\mathcal{Z}^0 \to \mathcal{Z}^1}(\hat z^0_i) \approx H_1^{-1}\left(H_0(\hat z^0_i)\right)
$$
where $H_s$ denotes the CDF for the demographic cohort $\mathcal{Z}^s$.
This expression is more scalable in the number of samples $|\mathcal{Z}|$, since it does not require to compute the coupling matrix $\gamma$. Note that this expression also holds for the transport map of \Cref{eq:mapping} re-written with $\varepsilon < 0$.

We now evaluate the quality of our derivatives estimates on a simulated dataset $\mathcal{Z}^0$ and $\mathcal{Z}^1$, with both 1000 samples, including duplicates.
For each sample $\hat{z}^0$ in the first group, we perform the following procedure to obtain an empirical ``oracle'' for its gradient and hessian:
\begin{enumerate}
    \item Let $\varepsilon > 0$ be a small perturbation  to be applied on the given sample $\hat{z}^0_c$.
    \item Let the perturbed dataset be $\mathcal{Z}^0_\varepsilon = \{\hat{z}^0_i\}_{i=1, i\neq c}^{n^0} \cup \{\hat{z}^0_c + \varepsilon\}$.
    \item Compute the squared Wasserstein distances for $\mathcal{Z}^0$ and $\mathcal{Z}^0_\varepsilon$, denoted $\mathcal{W}_2^2(\mathcal{Z}^0, \mathcal{Z}^1)$ and $\mathcal{W}_2^2(\mathcal{Z}^{0}_\varepsilon, \mathcal{Z}^1)$, respectively.
    \item Calculate the finite difference quotient: $E(\varepsilon) = \frac{\mathcal{W}_2^2(\mathcal{Z}^0_\varepsilon, \mathcal{Z}^1) - \mathcal{W}_2^2(\mathcal{Z}^0, \mathcal{Z}^1)}{\varepsilon}$, and return an average of $E(\varepsilon)$ over several small values of $\varepsilon$ for robustness.
\end{enumerate}
\begin{figure}[h]
    \centering    \includegraphics[width=0.8\linewidth]{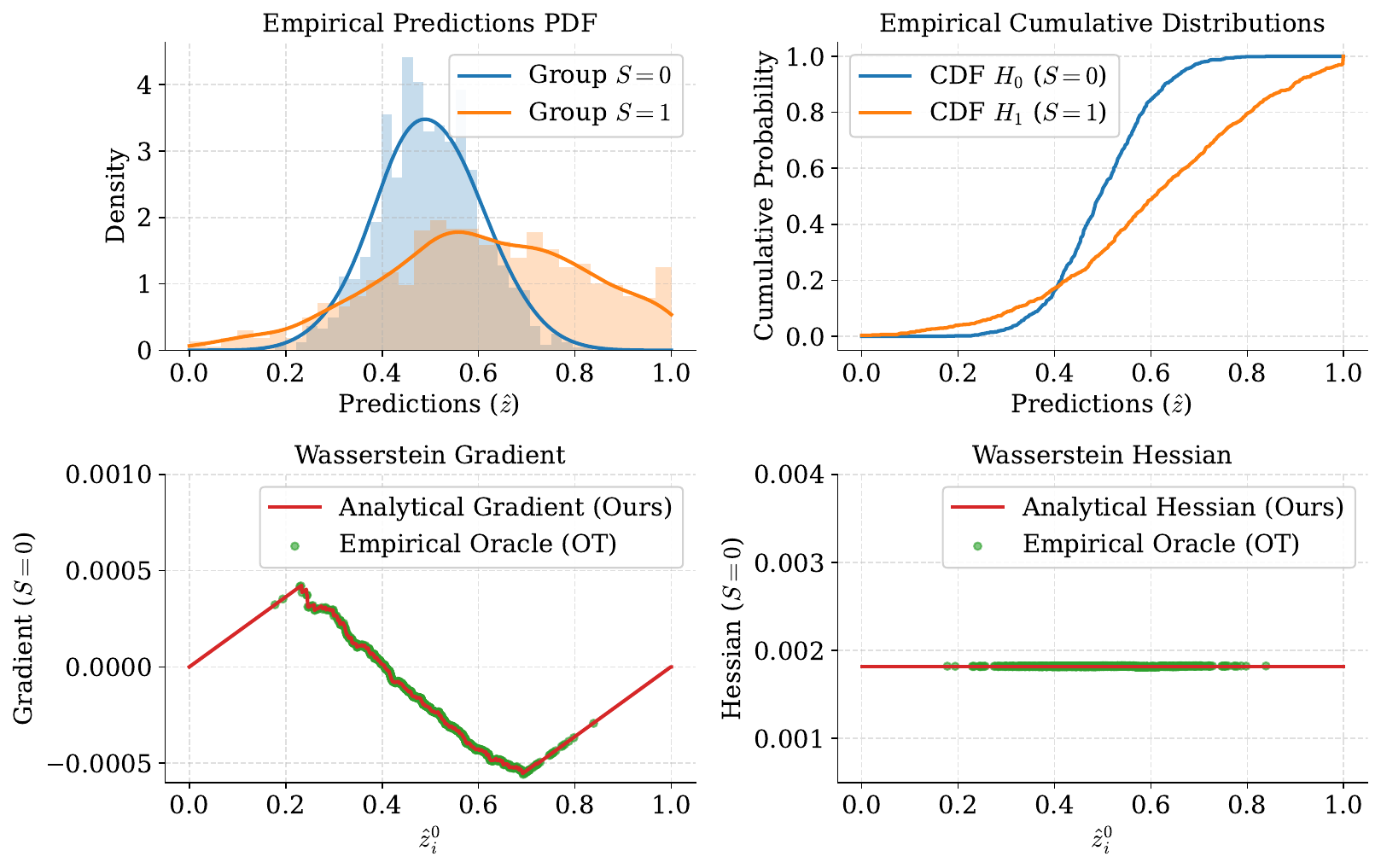}
    \caption {Simulated dataset distributions $\mathcal{Z}^0$ and $\mathcal{Z}^1$, and corresponding Wasserstein gradients and hessians for our estimate and the oracle on $\mathcal{Z}^0$ samples.}
    \label{fig:illustrations_estimates}
\end{figure}

The first row of \Cref{fig:illustrations_estimates} displays the probability and cumulative distribution functions (PDFs and CDFs) of the simulated groups, incorporating duplicated samples. The bottom row shows the Wasserstein gradients and hessians for each group, computed using the analytical formula from \Cref{proposition:main-derivatives}.
As observed in the figures, our analytical gradient and Hessian estimates align very well with the empirical oracle values for every sample, demonstrating the accuracy and effectiveness of our formulation.

We also illustrated in the Appendix that $\mathcal{W}_2^2(\mu^{\mathcal{Z}^0}, \mu^{\mathcal{Z}^1})$ is indeed minimized during OT-FairBoost boosting iterations.

%For the mapping function (Equation \ref{eq:mapping}) we instead consider the $\gamma$ computed directly on discrete populations with balanced allocations for duplicate values. Note that both converge when $N_0 \to \infty$, since the impact of moving $u^0_i$ on $gamma$ values is $\mathcal{O}(1/N_0)$. 

%We use pot package to compute optimal $\gamma$ between distributions.

\section{Experiments}
In this section, we integrate our Wasserstein estimates into the LightGBM training process.
%, defining a loss function that balances predictive performance and fairness via a trade-off parameter.

\paragraph{Datasets}
To evaluate the empirical performance and bias mitigation efficacy of OT-FairBoost, we conduct comprehensive evaluations across three standard benchmark datasets spanning varying scales, feature densities, and sensitive attribute tracking mechanics: Folktables \cite{ding2021retiring}, FairJob \cite{vladimirova2024fairjob} and Communities \& Crime \cite{redmond2002data}. A complete structural summary of these datasets is provided in \Cref{tab:datasets_summary}.

\begin{table*}[h]
\centering
\caption{Performances and fairness metrics comparison across different settings, when maximizing the objective $Obj_{\alpha=0.75}$. Post-processing baselines W1Post and W2Post are displayed in italics since they require $S$ as input for inference.}
\label{tab:fairness_metrics_main}
\setlength{\tabcolsep}{5pt} 
\small
\begin{tabular}{l|c|cccccc}
\toprule
\textbf{Metric} & \texttt{Baseline} & \texttt{FairGS} & \texttt{FairEG} & \texttt{FairGBM} & \texttt{OT-FairBoost} & \textit{\texttt{W1Post}}$^*$ & \textit{\texttt{W2Post}}$^*$ \\
\midrule
\multicolumn{8}{c}{Demographic Parity (DP) on Folktables Tennessee} \\
\midrule
PR AUC & 77.3$\pm$0.9 & 74.2$\pm$0.6 & 60.2$\pm$0.9  & -    & \textbf{75.6$\pm$1.1} & 70.7$\pm$0.8 & 74.3$\pm$1.0 \\
ROC AUC    & 88.5$\pm$0.6 & 86.9$\pm$0.4 & 78.4$\pm$0.4  & -    & \textbf{87.5$\pm$0.5} & 83.8$\pm$0.5 & 86.8$\pm$0.6  \\
$\Delta$DP    & 18.5$\pm$0.6 & 0.9$\pm$0.8 & 0.6$\pm$0.7  & -    & 1.4$\pm$1.2 & 2.4$\pm$1.1 & \textbf{0.0}$\pm$\textbf{0.0}  \\
$Obj_{\alpha=0.75}$ & 78.4$\pm$0.7 & 80.4$\pm$0.5 & 70.0$\pm$0.7  & -    & \textbf{81.4$\pm$0.9} & 77.4$\pm$0.7 & 80.7$\pm$0.7  \\
\midrule
\multicolumn{8}{c}{Equalized Odds (EOdds) Folktables Tennessee} \\
\midrule
PR AUC   & 77.3$\pm$0.9 &   77.1$\pm$1.0   &   60.6$\pm$1.4   & 76.2$\pm$0.7 &     \textbf{77.4$\pm$1.2}  & -          & -          \\
ROC AUC   & 88.5$\pm$0.6 &  \textbf{88.5$\pm$0.6}   &   78.8$\pm$0.8    & 88.1$\pm$0.5 &     \textbf{88.5$\pm$0.6}  & -          & -          \\
$\Delta$EOdds    & 22.7$\pm$2.0 &  22.2$\pm$3.1    &    7.4$\pm$3.6   & 9.2$\pm$2.2 &   \textbf{3.1$\pm$2.7}    & -          & -          \\
$Obj_{\alpha=0.75}$ & 77.3$\pm$0.9 &   77.2$\pm$1.1   &    68.6$\pm$1.4   & 79.8$\pm$0.8  & \textbf{82.3$\pm$1.1}      & -          & -          \\
\midrule
\multicolumn{8}{c}{Regression (DP) on Communities \& Crime} \\
\midrule
MAE      & 9.0$\pm$0.5      & -          & -          & -      & \textbf{13.3$\pm$1.2}      & 17.9$\pm$0.7      & 14.0$\pm$0.8      \\
KS  & 67.0$\pm$4.0      & -          & -          & -      & 9.7$\pm$3.9      & 75.8$\pm$ 4.0     & \textbf{5.9$\pm$2.1}      \\
$\mathcal{W}_2$  & 30.1$\pm$1.1     & -          & -          & -      & \textbf{1.4$\pm$0.5}      & 1.6$\pm$0.1      & \textbf{1.4$\pm$0.5}      \\
$Obj_{\alpha=0.75}$ & 85.6$\pm$0.5      & -          & -          & -      & \textbf{89.7$\pm$0.9}      & 86.2$\pm$0.5      & 89.1$\pm$0.6      \\
\midrule
\multicolumn{8}{c}{Multi-Group for Sensitive Attribute (DP) on Folktables Tennessee} \\
\midrule
PR AUC   & 76.9$\pm$0.4     & 72.9$\pm$0.5   & 62.2$\pm$0.5   & - & \textbf{76.3$\pm$0.4}      & 68.0$\pm$1.8      & 74.6$\pm$0.6      \\
ROC AUC   & 88.4$\pm$0.2    & 83.1$\pm$0.3   & 79.2$\pm$0.4   & - & \textbf{88.1$\pm$0.3}      & 82.7$\pm$1.2      & 87.2$\pm$0.4      \\
$\Delta_{\text{max}} \text{DP}$   & 20.2$\pm$3.2      & 34.5$\pm$0.9   & 5.4$\pm$0.9   & - & \textbf{4.2$\pm$2.2}      & 14.0$\pm$3.9      & 7.8$\pm$2.1      \\
$Obj_{\alpha=0.75}$ & 77.6$\pm$0.8    & 71.1$\pm$0.4   & 70.3$\pm$0.5   & - & \textbf{81.2$\pm$0.6}      & 72.4$\pm$1.7      & 79.0$\pm$0.7      \\
\bottomrule
\end{tabular}
\end{table*}
\paragraph{Evaluation Metrics} We report ROC AUC and PR AUC for classification and MAE for regression. Fairness is quantified by the empirical violation of the definitions in Section Model, measured by the following empirical metrics on the test set: %{sec:model} ref not working
\begin{itemize}
    \item For binary classification and binary sensitive variable
    \begin{equation}
    \label{eq:DP_gap_def}
        \Delta \text{DP} = |\mathbb{P}(\hat{Y} = 1 | S = 0) - \mathbb{P}(\hat{Y} = 1 | S = 1)|
    \end{equation}
    \begin{equation}
        \text{DI} = \frac{\min_{s \in \{0, 1\}}[\mathbb{P}(\hat{Y} = 1 | S = s)]}{\max_{s \in \{0, 1\}}[\mathbb{P}(\hat{Y} = 1 | S = s)]}
    \end{equation}
    \item For binary classification with a binary sensitive attribute, the Equalized Odds gap ($\Delta\text{EOdds}$) extends the Demographic Parity definition in \Cref{eq:DP_gap_def} by conditioning predictions on the ground-truth target $Y$. Specifically, it enforces rate equality across both target outcomes. %($Y=0$ and $Y=1$).
    \item For regression and binary sensitive variable, Wassertein-2 distance (see \Cref{eq:w_2_discrete}) and Kolmogorov-Smirnov distance following \cite{agarwal2019fair} 
    \begin{equation}
        D_{KS} = \sup_{\tau \in \mathbb{R}}|\mathbb{P}(\hat{Y} \leq \tau | S = 0) - \mathbb{P}(\hat{Y} \leq \tau | S = 1)|
        \nonumber
    \end{equation}
    \item For binary classification and multi-class sensitive variable
    \begin{equation}
        %\Delta_{\text{max}} \text{DP} = \max_{s \neq s'}|\mathbb{P}(\hat{Y} = 1 | S = s) - \mathbb{P}(\hat{Y} = 1 | S = s')|
        \Delta_{\text{max}} \text{DP} = \max_{s}|\mathbb{P}(\hat{Y} = 1 | S = s) - \mathbb{P}(\hat{Y} = 1 | S \neq s)|
        \nonumber
    \end{equation}
\end{itemize}
%These metrics are computed using empirical probabilities on the test set, and averaged across folds.
%These probabilities are estimated empirically on the test set in order to compute these metrics.

\paragraph{Protocol} We evaluate OT-FairBoost against alternative baselines through 5-fold cross-validation across an extensive array of hyperparameter configurations. Specifically, 100 random hyperparameter configurations is employed for each model to map their joint predictive capability and fairness profiles. Models including fairness optimization are indeed particularly sensitive to their hyperparameters \cite{cruz2021promoting}. Native LightGBM hyperparameters are sampled equally among methods (more details in Appendix).

\begin{figure}[h]
    \centering
    \includegraphics[width=0.6\linewidth]{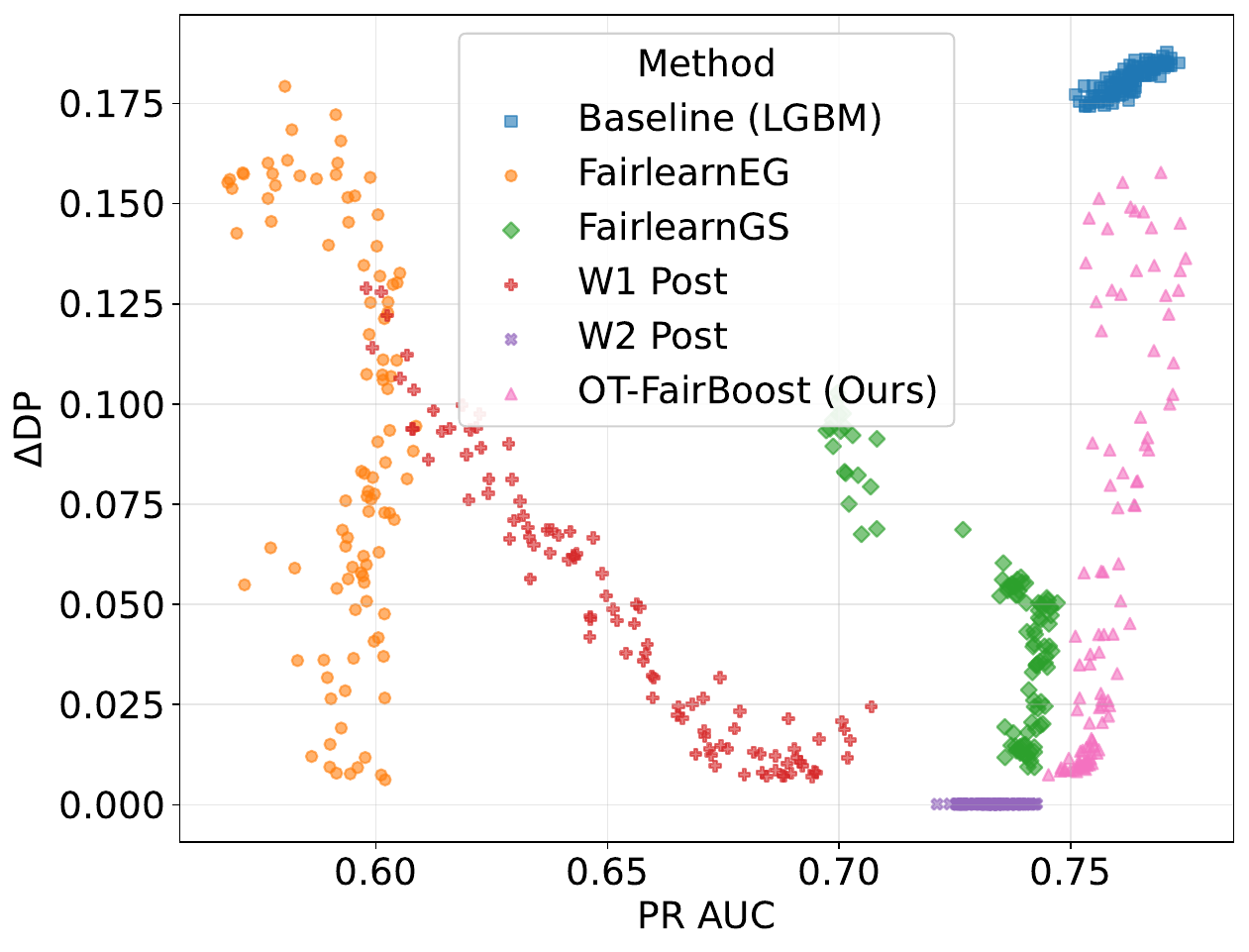}
    \caption{PR AUC vs. $\Delta$DP  of all models on Folktables Tennessee for sampled hyperparameters configurations.}
    \label{fig:several-configs-DP-folks-TN}
\end{figure}
To quantitatively evaluate this performance-fairness equilibrium, we use a scalarized multi-objective criterion introduced by \citet{cruz2022fairgbm}, $Obj_\alpha$, parameterized by a preference weight $\alpha \in [0, 1]$:
\begin{equation}
    Obj_\alpha = \alpha \cdot \text{Performance} + (1-\alpha) \cdot (1 - \text{Fairness}),
\end{equation}
where higher values signify a better trade-off. The underlying metrics are task-dependent: predictive performance is evaluated using PR AUC for binary classification and MAE for regression. For classification tasks, fairness is quantified via the Demographic Parity gap ($\Delta$DP) and Equalized Odds gap ($\Delta$EOdds). For regression, fairness is evaluated using the Wasserstein-2 distance between the group-conditional prediction distributions.
We report in \Cref{tab:fairness_metrics_main} the metrics for the configuration maximizing this objective for $\alpha=0.75$, prioritizing predictive utility while penalizing bias. FairGBM is not designed for DP, thus absent from this Table.
%These optimized configurations are reported in the upper section of \Cref{tab:fairness_metrics_main}.

\begin{table*}[t]
\centering
\caption{Binary classification on large-scale FairJob dataset with optimization on DP. $\lambda=10.0$ for OT-FairBoost.}
\label{tab:fairness_metrics_fairjob}
\setlength{\tabcolsep}{4pt} 
\small
\begin{tabular}{l|c|cccccc}
\toprule
\textbf{Metric} & \texttt{Baseline} & \texttt{FairGS} & \texttt{FairEG} &  \texttt{OT-FairBoost} & \textit{\texttt{W1Post}}$^*$ & \textit{\texttt{W2Post}}$^*$ \\
\midrule
PR AUC  & 4.3 $\pm$ 0.2   & 4.0 $\pm$ 0.3   & 1.9 $\pm$ 0.1    & \textbf{4.5 $\pm$ 0.3}   & 3.7 $\pm$ 0.2   & 4.2 $\pm$ 0.2   \\
ROC AUC & 79.4 $\pm$ 0.7  & 79.3 $\pm$ 0.6  & 64.9 $\pm$ 0.5 & \textbf{82.0 $\pm$ 0.5}  & 76.3 $\pm$ 1.3  & 79.1 $\pm$ 0.6  \\
$\Delta$DP  & 1.35 $\pm$ 0.05  & 1.49 $\pm$ 0.11  & 0.93 $\pm$ 0.30   & 1.3 $\pm$ 0.12            & \textbf{0.86 $\pm$ 0.24}  & 0.99 $\pm$ 0.28  \\
DI      & 80.98 $\pm$ 0.63 & 77.97 $\pm$ 2.09 & 86.28 $\pm$ 4.16  & \textbf{94.33 $\pm$ 0.57} & 86.82 $\pm$ 3.22 & 85.10 $\pm$ 3.67 \\
$Obj_{\alpha=0.75}$ & 27.92 $\pm$ 0.16 & 27.63 $\pm$ 0.23 & 26.17 $\pm$ 0.04  & \textbf{28.03 $\pm$ 0.22} & 27.56 $\pm$ 0.10 & 27.90 $\pm$ 0.17 \\
\midrule
Time(s) & 12.3 $\pm$ 0.4   & 149.7 $\pm$ 21.5 & 508.8 $\pm$ 47.3 & 243.2 $\pm$ 6.5           & 861.4 $\pm$ 108.5 & 244.8 $\pm$ 8.4   \\
\bottomrule
\end{tabular}
\end{table*}
\paragraph{Baselines}
%\textcolor{red}{Focus only on LightGBM variants.}
To ensure a fair and controlled comparison, our benchmark evaluates only those procedures that either natively extend LightGBM or operate as model-agnostic methods directly compatible with it. Thus, we compare our proposed procedure against an unconstrained LightGBM model, referred to as Baseline. For in-processing benchmarks, we include FairEG \cite{agarwal2018reductions},  FairGS \cite{agarwal2018reductions}, and FairGBM \cite{cruz2022fairgbm}. Additionally, we evaluate two post-processing methods, W1Post \cite{xian2023fair} and W2Post \cite{gouic2020projection}. Note that these post-processing techniques require the sensitive attribute value at inference time, a requirement that in-processing frameworks, including ours, explicitly avoid. %Consequently, in \Cref{tab:fairness_metrics_main}, the post-processing baselines W1Post and W2Post are displayed in italics on the right.

%% Seeting classif binaire+DP+ sensitive binaire. (scatter plot)
\paragraph{Results on DP}
We adopt this protocol to assess methodologies optimizing Demographic Parity (DP) on Folktables Tennessee data-set. As illustrated in \Cref{fig:several-configs-DP-folks-TN}, OT-FairBoost consistently occupies the lower-right quadrant of the objective space. This position demonstrates a superior Pareto frontier, as our method simultaneously maximizes the PR AUC while minimizing demographic disparity relative to competing baselines.
The empirical results demonstrate that OT-FairBoost achieves the highest overall objective score 81.4, highlighting its ability to maintain high predictive accuracy under stringent fairness constraints. %Notably, post-processing methods like W2Post also yield strong results 80.7.

\begin{figure}[h]
    \centering
    \includegraphics[width=0.8\linewidth]{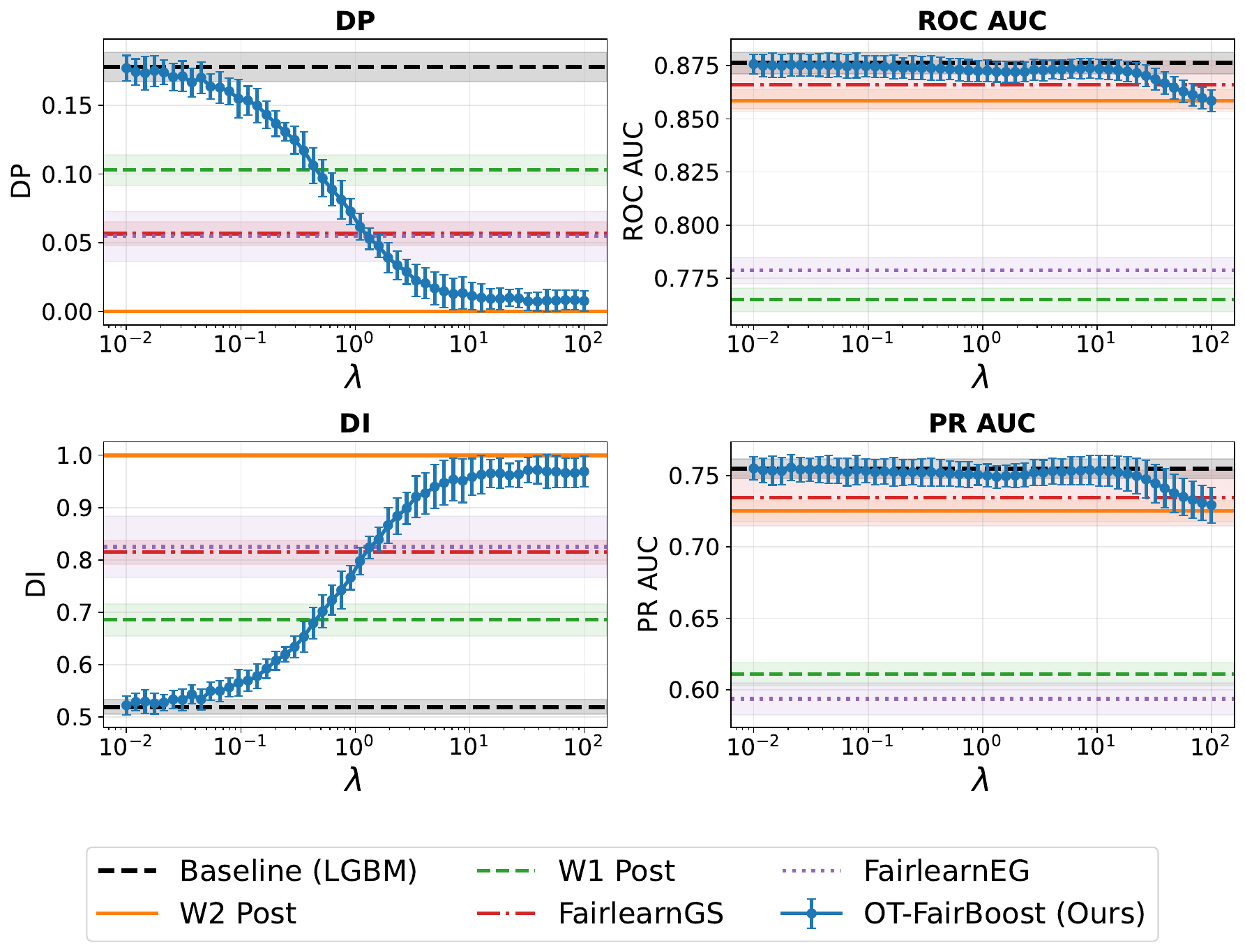}
    \caption{OT-FairBoost metrics with respect to $\lambda$ compared to other methods (other hyperparameters fixed).}
    \label{fig:tennessee_lambda_metrics}
\end{figure}

\paragraph{Trade-off between fairness and accuracy}
Finally, we analyze the sensitivity and operational stability of the proposed framework by tracking its behavior under a fixed hyperparameter configuration while varying the training fairness penalty parameter $\lambda$. The resulting trajectories are depicted in \Cref{fig:tennessee_lambda_metrics}. We observe that OT-FairBoost exhibits a highly smooth and monotonic behavior; as the regularization strength $\lambda$ scales up, demographic parity decreases systematically, incurring only a marginal, well-behaved degradation in ROC and PR AUC. This predictable compliance behavior underscores the practical viability of our method for real-world deployments.

\paragraph{Extension to Equalized Odds, Regression, Multi-Group Attributes and large-scale dataset}

To demonstrate the algorithmic versatility and architectural flexibility of OT-FairBoost, we conducted experiments following our protocol across three distinct operational objectives:
\begin{itemize}
    \item Equalized Odds (EOdds) in binary classification: evaluated on the Folktables Tennessee dataset using binary sex as the sensitive attribute.
    \item Demographic Parity (DP) in regression: evaluated on the Communities \& Crime dataset, where the sensitive feature is constructed by binarizing the proportion of Black residents at a threshold of $> 0.23$.
    \item Multi-group Demographic Parity (DP) in classification: evaluated on the Folktables Tennessee dataset using multi-category ethnicity as the sensitive attribute.
\end{itemize}

The consolidated comparative results, highlighting the optimal configuration that maximizes the scalarized trade-off objective for each baseline, are reported in \Cref{tab:fairness_metrics_main}. Their corresponding empirical plots are provided in the Appendix.

Across Equalized Odds, Regression DP, and Multi-Group DP, OT-FairBoost consistently yields the superior performance-fairness trade-off by having the highest trade-off objective value. These results demonstrate that, beyond its versatility across diverse settings, our proposed methodology remains highly competitive.

We further evaluate our approach against baselines on the large-scale, real-world FairJob dataset optimizing DP. Each methodology is evaluated via $5$-fold cross-validation using default hyperparameter settings. Due to the extreme class imbalance in this dataset, sample weighting is applied across all classifiers to obtain meaningful $\Delta\text{DP}$ and Disparate Impact ($\text{DI}$) metrics. The predictive performance, fairness metrics, and training runtimes are summarized in \Cref{tab:fairness_metrics_fairjob}. Notably, OT-FairBoost achieves the highest composite objective score of $28.05$, outperforming competitors, while maintaining a runtime comparable to W2Post, the strongest competitor.

%\textcolor{red}{Mentionner appendix davantage de figures}

\section{Conclusion}

In this work, we introduced OT-FairBoost, a novel in-processing algorithmic framework that leverages Optimal Transport to enforce fairness constraints within gradient-boosted decision trees. To do so, we derived samplewise gradients and hessians from the discrete Wasserstein-2 distance between groups' predictions, that are confirmed experimentally on oracle estimates.
When integrated in lightGBM training procedure, empirical results demonstrate that OT-FairBoost consistently achieves a superior trade-off between predictive performance and fairness, providing a scalable and controllable solution for trustworthy machine learning on tabular data.
Furthermore, OT-FairBoost flexibly adapts to standard group fairness criteria, addressing Demographic Parity as well as Equalized Odds, both in the case of classification and regression, with binary or multigroup sensitive attribute.
%\textcolor{red}{Future work ?}
%Empirical results demonstrate that our method consistently achieves a superior and operationally stable Pareto frontier, providing a scalable and controllable solution for trustworthy machine learning on tabular data.

% je censure ? Pareto frontier

\clearpage
\bibliographystyle{plainnat}
\bibliography{refs}

\begin{thebibliography}{53}
\providecommand{\natexlab}[1]{#1}
\providecommand{\url}[1]{\texttt{#1}}
\expandafter\ifx\csname urlstyle\endcsname\relax
  \providecommand{\doi}[1]{doi: #1}\else
  \providecommand{\doi}{doi: \begingroup \urlstyle{rm}\Url}\fi

\bibitem[Agarwal et~al.(2018)Agarwal, Beygelzimer, Dud{\'\i}k, Langford, and Wallach]{agarwal2018reductions}
Alekh Agarwal, Alina Beygelzimer, Miroslav Dud{\'\i}k, John Langford, and Hanna Wallach.
\newblock A reductions approach to fair classification.
\newblock In \emph{International conference on machine learning}, pages 60--69. PMLR, 2018.

\bibitem[Agarwal et~al.(2019)Agarwal, Dud{\'\i}k, and Wu]{agarwal2019fair}
Alekh Agarwal, Miroslav Dud{\'\i}k, and Zhiwei~Steven Wu.
\newblock Fair regression: Quantitative definitions and reduction-based algorithms.
\newblock In \emph{International conference on machine learning}, pages 120--129. PMLR, 2019.

\bibitem[Agarwal and Deshpande(2022)]{agarwal2022power}
Sushant Agarwal and Amit Deshpande.
\newblock On the power of randomization in fair classification and representation.
\newblock In \emph{Proceedings of the 2022 ACM Conference on Fairness, Accountability, and Transparency}, pages 1542--1551, 2022.

\bibitem[Besse et~al.(2022)Besse, del Barrio, Gordaliza, Loubes, and Risser]{Besse03042022}
Philippe Besse, Eustasio del Barrio, Paula Gordaliza, Jean-Michel Loubes, and Laurent Risser.
\newblock A survey of bias in machine learning through the prism of statistical parity.
\newblock \emph{The American Statistician}, 76\penalty0 (2):\penalty0 188--198, 2022.
\newblock \doi{10.1080/00031305.2021.1952897}.

\bibitem[Calmon et~al.(2017)Calmon, Wei, Vinzamuri, Natesan~Ramamurthy, and Varshney]{calmon2017optimized}
Flavio Calmon, Dennis Wei, Bhanukiran Vinzamuri, Karthikeyan Natesan~Ramamurthy, and Kush~R Varshney.
\newblock Optimized pre-processing for discrimination prevention.
\newblock \emph{Advances in neural information processing systems}, 30, 2017.

\bibitem[Chen and Guestrin(2016)]{chen2016xgboost}
Tianqi Chen and Carlos Guestrin.
\newblock Xgboost: A scalable tree boosting system.
\newblock In \emph{Proceedings of the 22nd acm sigkdd international conference on knowledge discovery and data mining}, pages 785--794, 2016.

\bibitem[Chouldechova(2017)]{chouldechova2017fair}
Alexandra Chouldechova.
\newblock Fair prediction with disparate impact: A study of bias in recidivism prediction instruments.
\newblock \emph{Big data}, 5\penalty0 (2):\penalty0 153--163, 2017.

\bibitem[Chzhen et~al.(2020)Chzhen, Denis, Hebiri, Oneto, and Pontil]{chzhen2020fair}
Evgenii Chzhen, Christophe Denis, Mohamed Hebiri, Luca Oneto, and Massimiliano Pontil.
\newblock Fair regression with wasserstein barycenters.
\newblock \emph{Advances in Neural Information Processing Systems}, 33:\penalty0 7321--7331, 2020.

\bibitem[Corbett-Davies and Goel(2018)]{corbett2018measure}
Sam Corbett-Davies and Sharad Goel.
\newblock The measure and mismeasure of fairness: A critical review of fair machine learning.
\newblock \emph{arXiv preprint arXiv:1808.00023}, 14, 2018.

\bibitem[Cotter et~al.(2019)Cotter, Jiang, Gupta, Wang, Narayan, You, and Sridharan]{cotter2019optimization}
Andrew Cotter, Heinrich Jiang, Maya Gupta, Serena Wang, Taman Narayan, Seungil You, and Karthik Sridharan.
\newblock Optimization with non-differentiable constraints with applications to fairness, recall, churn, and other goals.
\newblock \emph{Journal of Machine Learning Research}, 20\penalty0 (172):\penalty0 1--59, 2019.

\bibitem[Cruz et~al.(2021)Cruz, Saleiro, Bel{\'e}m, Soares, and Bizarro]{cruz2021promoting}
Andr{\'e}~F Cruz, Pedro Saleiro, Catarina Bel{\'e}m, Carlos Soares, and Pedro Bizarro.
\newblock Promoting fairness through hyperparameter optimization.
\newblock In \emph{2021 IEEE international conference on data mining (ICDM)}, pages 1036--1041. IEEE, 2021.

\bibitem[Cruz et~al.(2022)Cruz, Bel{\'e}m, Jesus, Bravo, Saleiro, and Bizarro]{cruz2022fairgbm}
Andr{\'e}~F Cruz, Catarina Bel{\'e}m, S{\'e}rgio Jesus, Jo{\~a}o Bravo, Pedro Saleiro, and Pedro Bizarro.
\newblock Fairgbm: Gradient boosting with fairness constraints.
\newblock \emph{arXiv preprint arXiv:2209.07850}, 2022.

\bibitem[Ding et~al.(2021)Ding, Hardt, Miller, and Schmidt]{ding2021retiring}
Frances Ding, Moritz Hardt, John Miller, and Ludwig Schmidt.
\newblock Retiring adult: New datasets for fair machine learning.
\newblock \emph{Advances in neural information processing systems}, 34:\penalty0 6478--6490, 2021.

\bibitem[Donini et~al.(2018)Donini, Oneto, Ben-David, Shawe-Taylor, and Pontil]{donini2018empirical}
Michele Donini, Luca Oneto, Shai Ben-David, John~S Shawe-Taylor, and Massimiliano Pontil.
\newblock Empirical risk minimization under fairness constraints.
\newblock \emph{Advances in neural information processing systems}, 31, 2018.

\bibitem[Dwork et~al.(2012)Dwork, Hardt, Pitassi, Reingold, and Zemel]{dwork2012fairness}
Cynthia Dwork, Moritz Hardt, Toniann Pitassi, Omer Reingold, and Richard Zemel.
\newblock Fairness through awareness.
\newblock In \emph{Proceedings of the 3rd innovations in theoretical computer science conference}, pages 214--226, 2012.

\bibitem[Feldman et~al.(2015)Feldman, Friedler, Moeller, Scheidegger, and Venkatasubramanian]{feldman2015certifying}
Michael Feldman, Sorelle~A Friedler, John Moeller, Carlos Scheidegger, and Suresh Venkatasubramanian.
\newblock Certifying and removing disparate impact.
\newblock In \emph{proceedings of the 21th ACM SIGKDD international conference on knowledge discovery and data mining}, pages 259--268, 2015.

\bibitem[Gaucher et~al.(2023)Gaucher, Schreuder, and Chzhen]{gaucher2023fair}
Solenne Gaucher, Nicolas Schreuder, and Evgenii Chzhen.
\newblock Fair learning with wasserstein barycenters for non-decomposable performance measures.
\newblock In \emph{International Conference on Artificial Intelligence and Statistics}, pages 2436--2459. PMLR, 2023.

\bibitem[Gordaliza et~al.(2019)Gordaliza, Del~Barrio, Fabrice, and Loubes]{gordaliza2019obtaining}
Paula Gordaliza, Eustasio Del~Barrio, Gamboa Fabrice, and Jean-Michel Loubes.
\newblock Obtaining fairness using optimal transport theory.
\newblock In \emph{International conference on machine learning}, pages 2357--2365. PMLR, 2019.

\bibitem[Gouic et~al.(2020)Gouic, Loubes, and Rigollet]{gouic2020projection}
Thibaut~Le Gouic, Jean-Michel Loubes, and Philippe Rigollet.
\newblock Projection to fairness in statistical learning.
\newblock \emph{arXiv preprint arXiv:2005.11720}, 2020.

\bibitem[Grari et~al.(2019)Grari, Ruf, Lamprier, and Detyniecki]{grari2019fair}
Vincent Grari, Boris Ruf, Sylvain Lamprier, and Marcin Detyniecki.
\newblock Fair adversarial gradient tree boosting.
\newblock In \emph{2019 IEEE international conference on data mining (ICDM)}, pages 1060--1065. IEEE, 2019.

\bibitem[Grinsztajn et~al.(2022)Grinsztajn, Oyallon, and Varoquaux]{grinsztajn2022tree}
L{\'e}o Grinsztajn, Edouard Oyallon, and Ga{\"e}l Varoquaux.
\newblock Why do tree-based models still outperform deep learning on typical tabular data?
\newblock \emph{Advances in neural information processing systems}, 35:\penalty0 507--520, 2022.

\bibitem[Hardt et~al.(2016)Hardt, Price, and Srebro]{hardt2016equality}
Moritz Hardt, Eric Price, and Nati Srebro.
\newblock Equality of opportunity in supervised learning.
\newblock \emph{Advances in neural information processing systems}, 29, 2016.

\bibitem[Hollmann et~al.(2025)Hollmann, M{\"u}ller, Purucker, Krishnakumar, K{\"o}rfer, Hoo, Schirrmeister, and Hutter]{hollmann2025accurate}
Noah Hollmann, Samuel M{\"u}ller, Lennart Purucker, Arjun Krishnakumar, Max K{\"o}rfer, Shi~Bin Hoo, Robin~Tibor Schirrmeister, and Frank Hutter.
\newblock Accurate predictions on small data with a tabular foundation model.
\newblock \emph{Nature}, 637\penalty0 (8045):\penalty0 319--326, 2025.

\bibitem[Jiang et~al.(2020)Jiang, Pacchiano, Stepleton, Jiang, and Chiappa]{jiang2020wasserstein}
Ray Jiang, Aldo Pacchiano, Tom Stepleton, Heinrich Jiang, and Silvia Chiappa.
\newblock Wasserstein fair classification.
\newblock In \emph{Uncertainty in artificial intelligence}, pages 862--872. PMLR, 2020.

\bibitem[Kamiran and Calders(2009)]{kamiran2009classifying}
Faisal Kamiran and Toon Calders.
\newblock Classifying without discriminating.
\newblock In \emph{2009 2nd international conference on computer, control and communication}, pages 1--6. IEEE, 2009.

\bibitem[Kamiran and Calders(2012)]{kamiran2012data}
Faisal Kamiran and Toon Calders.
\newblock Data preprocessing techniques for classification without discrimination.
\newblock \emph{Knowledge and Information Systems}, 33\penalty0 (1):\penalty0 1--33, 2012.

\bibitem[Kamiran et~al.(2012)Kamiran, Karim, and Zhang]{kamiran2012decision}
Faisal Kamiran, Asim Karim, and Xiangliang Zhang.
\newblock Decision theory for discrimination-aware classification.
\newblock In \emph{2012 IEEE 12th international conference on data mining}, pages 924--929. IEEE, 2012.

\bibitem[Ke et~al.(2017)Ke, Meng, Finley, Wang, Chen, Ma, Ye, and Liu]{ke2017lightgbm}
Guolin Ke, Qi~Meng, Thomas Finley, Taifeng Wang, Wei Chen, Weidong Ma, Qiwei Ye, and Tie-Yan Liu.
\newblock Lightgbm: A highly efficient gradient boosting decision tree.
\newblock \emph{Advances in neural information processing systems}, 30, 2017.

\bibitem[Khalilia et~al.(2011)Khalilia, Chakraborty, and Popescu]{RF-medical-example}
Mohammed Khalilia, Sounak Chakraborty, and Mihail Popescu.
\newblock Predicting disease risks from highly imbalanced data using random forest.
\newblock \emph{BMC medical informatics and decision making}, 11\penalty0 (1):\penalty0 51, 2011.

\bibitem[Kohavi and Becker(1996)]{adultCensus96}
R.~Kohavi and B.~Becker.
\newblock Uci adult data set.
\newblock UCI Meachine Learning Repository, 5 1996.
\newblock DOI: 10.24432/C5XW20.

\bibitem[Louizos et~al.(2015)Louizos, Swersky, Li, Welling, and Zemel]{louizos2015variational}
Christos Louizos, Kevin Swersky, Yujia Li, Max Welling, and Richard Zemel.
\newblock The variational fair autoencoder.
\newblock \emph{arXiv preprint arXiv:1511.00830}, 2015.

\bibitem[Mehrabi et~al.(2021)Mehrabi, Morstatter, Saxena, Lerman, and Galstyan]{mehrabi2021survey}
Ninareh Mehrabi, Fred Morstatter, Nripsuta Saxena, Kristina Lerman, and Aram Galstyan.
\newblock A survey on bias and fairness in machine learning.
\newblock \emph{ACM computing surveys (CSUR)}, 54\penalty0 (6):\penalty0 1--35, 2021.

\bibitem[Nguyen and Duong(2021)]{nguyen2021comparison}
Nam~N Nguyen and Anh~T Duong.
\newblock Comparison of two main approaches for handling imbalanced data in churn prediction problem.
\newblock \emph{Journal of advances in information technology}, 12\penalty0 (1), 2021.

\bibitem[Peyr{\'e} and Cuturi(2019)]{peyre2019computational}
Gabriel Peyr{\'e} and Marco Cuturi.
\newblock \emph{Computational optimal transport: With applications to data science}.
\newblock Now Foundations and Trends, 2019.

\bibitem[Qu et~al.(2025)Qu, Holzmuller, Varoquaux, and Morvan]{qu2025tabicl}
Jingang Qu, David Holzmuller, Gael Varoquaux, and Marine~Le Morvan.
\newblock Tabicl: A tabular foundation model for in-context learning on large data.
\newblock \emph{arXiv preprint arXiv:2502.05564}, 2025.

\bibitem[Qu et~al.(2026)Qu, Holzmuller, Varoquaux, and Morvan]{qu2026tabiclv2}
Jingang Qu, David Holzmuller, Gael Varoquaux, and Marine~Le Morvan.
\newblock Tabiclv2: A better, faster, scalable, and open tabular foundation model.
\newblock \emph{arXiv preprint arXiv:2602.11139}, 2026.

\bibitem[Ravichandran et~al.(2020)Ravichandran, Khurana, Venkatesh, and Edakunni]{ravichandran2020fairxgboost}
Srinivasan Ravichandran, Drona Khurana, Bharath Venkatesh, and Narayanan~Unny Edakunni.
\newblock Fairxgboost: Fairness-aware classification in xgboost.
\newblock \emph{arXiv preprint arXiv:2009.01442}, 2020.

\bibitem[Redmond and Baveja(2002)]{redmond2002data}
Michael Redmond and Alok Baveja.
\newblock A data-driven software tool for enabling cooperative information sharing among police departments.
\newblock \emph{European Journal of Operational Research}, 141\penalty0 (3):\penalty0 660--678, 2002.

\bibitem[{Regulation (EU) 2024/1689}(2024)]{aiact}
{Regulation (EU) 2024/1689}.
\newblock Regulation (eu) 2024/1689 of the european parliament and of the council of 13 june 2024 laying down harmonised rules on artificial intelligence (artificial intelligence act).
\newblock Official Journal of the European Union, L 2024/1689., 2024.
\newblock https://eur-lex.europa.eu/eli/reg/2024/1689/oj/eng.

\bibitem[Risser et~al.(2022)Risser, Sanz, Vincenot, and Loubes]{risser2022tackling}
Laurent Risser, Alberto~Gonzalez Sanz, Quentin Vincenot, and Jean-Michel Loubes.
\newblock Tackling algorithmic bias in neural-network classifiers using wasserstein-2 regularization.
\newblock \emph{Journal of Mathematical Imaging and Vision}, 64\penalty0 (6):\penalty0 672--689, 2022.

\bibitem[Sakho et~al.(2025)Sakho, Malherbe, Gauthier, and Scornet]{sakho2025harnessing}
Abdoulaye Sakho, Emmanuel Malherbe, Carl-Erik Gauthier, and Erwan Scornet.
\newblock Harnessing mixed features for imbalance data oversampling: Application to bank customers scoring.
\newblock In \emph{Joint European Conference on Machine Learning and Knowledge Discovery in Databases}, pages 247--264. Springer, 2025.

\bibitem[Shilova et~al.(2025)Shilova, Malherbe, Palma, Risser, and Loubes]{shilova2025fairness}
Veronika Shilova, Emmanuel Malherbe, Giovanni Palma, Laurent Risser, and Jean-Michel Loubes.
\newblock Fairness-aware grouping for continuous sensitive variables: Application for debiasing face analysis with respect to skin tone.
\newblock \emph{arXiv preprint arXiv:2507.11247}, 2025.

\bibitem[Shwartz-Ziv and Armon(2022)]{shwartz2022tabular}
Ravid Shwartz-Ziv and Amitai Armon.
\newblock Tabular data: Deep learning is not all you need.
\newblock \emph{Information fusion}, 81:\penalty0 84--90, 2022.

\bibitem[Sun et~al.(2009)Sun, Wong, and Kamel]{sun2009classification}
Yanmin Sun, Andrew~KC Wong, and Mohamed~S Kamel.
\newblock Classification of imbalanced data: A review.
\newblock \emph{International journal of pattern recognition and artificial intelligence}, 23\penalty0 (04):\penalty0 687--719, 2009.

\bibitem[Verma and Rubin(2018)]{verma2018fairness}
Sahil Verma and Julia Rubin.
\newblock Fairness definitions explained.
\newblock In \emph{Proceedings of the international workshop on software fairness}, pages 1--7, 2018.

\bibitem[Vladimirova et~al.(2024)Vladimirova, Diemert, and Pavone]{vladimirova2024fairjob}
Mariia Vladimirova, Eustache Diemert, and Federico Pavone.
\newblock Fairjob: A real-world dataset for fairness in online systems.
\newblock \emph{Advances in Neural Information Processing Systems}, 37:\penalty0 10442--10469, 2024.

\bibitem[Wei et~al.(2020)Wei, Ramamurthy, and Calmon]{wei2020optimized}
Dennis Wei, Karthikeyan~Natesan Ramamurthy, and Flavio~P Calmon.
\newblock Optimized score transformation for fair classification.
\newblock In \emph{AISTATS}, volume~20, pages 1673--1683, 2020.

\bibitem[Xian et~al.(2023)Xian, Yin, and Zhao]{xian2023fair}
Ruicheng Xian, Lang Yin, and Han Zhao.
\newblock Fair and optimal classification via post-processing.
\newblock In \emph{International conference on machine learning}, pages 37977--38012. PMLR, 2023.

\bibitem[Xu et~al.(2018)Xu, Yuan, Zhang, and Wu]{xu2018fairgan}
Depeng Xu, Shuhan Yuan, Lu~Zhang, and Xintao Wu.
\newblock Fairgan: Fairness-aware generative adversarial networks.
\newblock In \emph{2018 IEEE international conference on big data (big data)}, pages 570--575. IEEE, 2018.

\bibitem[Zafar et~al.(2017)Zafar, Valera, Rogriguez, and Gummadi]{zafar2017fairness}
Muhammad~Bilal Zafar, Isabel Valera, Manuel~Gomez Rogriguez, and Krishna~P Gummadi.
\newblock Fairness constraints: Mechanisms for fair classification.
\newblock In \emph{Artificial intelligence and statistics}, pages 962--970. PMLR, 2017.

\bibitem[Zemel et~al.(2013)Zemel, Wu, Swersky, Pitassi, and Dwork]{zemel2013learning}
Rich Zemel, Yu~Wu, Kevin Swersky, Toni Pitassi, and Cynthia Dwork.
\newblock Learning fair representations.
\newblock In \emph{International conference on machine learning}, pages 325--333. PMLR, 2013.

\bibitem[Zhang et~al.(2018)Zhang, Lemoine, and Mitchell]{zhang2018mitigating}
Brian~Hu Zhang, Blake Lemoine, and Margaret Mitchell.
\newblock Mitigating unwanted biases with adversarial learning.
\newblock In \emph{Proceedings of the 2018 AAAI/ACM Conference on AI, Ethics, and Society}, pages 335--340, 2018.

\bibitem[{\v{Z}}liobait{\.e}(2017)]{vzliobaite2017measuring}
Indr{\.e} {\v{Z}}liobait{\.e}.
\newblock Measuring discrimination in algorithmic decision making.
\newblock \emph{Data Mining and Knowledge Discovery}, 31\penalty0 (4):\penalty0 1060--1089, 2017.

\end{thebibliography}

%\section*{Acknowledgments}
%Blabla ????

\clearpage
\appendix
\clearpage
\section{Details on numerical experiments}
All experiments were conducted on an AMD Ryzen Threadripper PRO 5955WX workstation (16 cores, 4.0\,GHz) equipped with 256\,GB of RAM.

The implementation details for our 5-fold cross-validation protocol are structured as follows:
\begin{itemize}
    \item Preprocessing: Categorical features in the Folktables dataset are encoded using a \texttt{OneHotEncoder}.
\end{itemize}

The shared LightGBM hyperparameter search space used across all models is detailed in \Cref{tab:lgbm_hyperparameters}, whereas method-specific hyperparameter configurations are summarized in \Cref{tab:method_hyperparameters}.
\begin{table}[h]
\centering
\caption{LightGBM hyperparameter search space and fixed experimental settings.}
\label{tab:lgbm_hyperparameters}
\small
\begin{tabular}{llcc}
\toprule
\multicolumn{4}{c}{\textbf{Varied Hyperparameters (Search Space)}} \\
\midrule
\textbf{Parameter} & \textbf{Type} & \textbf{Search Range} & \textbf{Scale} \\
\midrule
\texttt{num\_leaves} & Integer & $[16, 64]$ & Uniform \\
\texttt{max\_depth} & Integer & $[4, 10]$ & Uniform \\
\texttt{reg\_lambda} & Float & $[10^{-10}, 10^{-1}]$ & Log-uniform \\
\texttt{feature\_fraction} & Float & $[0.7, 1.0]$ & Uniform \\
\texttt{bagging\_fraction} & Float & $[0.7, 1.0]$ & Uniform \\
\midrule
\multicolumn{4}{c}{\textbf{Fixed Hyperparameters}} \\
\midrule
\textbf{Parameter} & \textbf{Value} & \multicolumn{2}{l}{\textbf{Description / Default}} \\
\midrule
\texttt{boosting\_type} & \texttt{"gbdt"} & \multicolumn{2}{l}{Gradient Boosted Decision Trees} \\
\texttt{n\_estimators} & $1000$ & \multicolumn{2}{l}{Number of boosting iterations} \\
\texttt{learning\_rate} & $0.1$ & \multicolumn{2}{l}{Shrinkage rate} \\
\texttt{reg\_alpha} & $0.0$ & \multicolumn{2}{l}{$L_1$ regularization} \\
\texttt{bagging\_freq} & $1$ & \multicolumn{2}{l}{Bagging frequency} \\
\texttt{enable\_bundle} & \texttt{True} & \multicolumn{2}{l}{Exclusive feature bundling} \\
\texttt{verbosity} & $-1$ & \multicolumn{2}{l}{Silent logging} \\
\texttt{seed} & \texttt{random\_seed} & \multicolumn{2}{l}{Fixed for reproducibility} \\
\bottomrule
\end{tabular}
\end{table}

\begin{table}[h]
\centering
\caption{Method-specific hyperparameter search spaces.}
\label{tab:method_hyperparameters}
\small
\begin{tabular}{llccc}
\toprule
\textbf{Method} & \textbf{Parameter} & \textbf{Type} & \textbf{Search Range} & \textbf{Sampling Scale} \\
\midrule
\multirow{2}{*}{\texttt{OT-FairBoost} (Ours)} & Fairness penalty ($\lambda$) & Float & $[0.1, 20.0]$ & Log-uniform \\
 & KDE Bandwidth & Float & $[0.01, 2.0]$ & Uniform \\
\midrule
\textit{\texttt{W1Post}} / \textit{\texttt{W2Post}} & Penalty weight & Float & $[10^{-4}, 0.05]$ & Uniform \\
\midrule
\texttt{FairLearn} (\texttt{EG} / \texttt{GS}) & DP Bound & Float & $[0.005, 0.1]$ & Uniform \\
\midrule
\texttt{FairGBM} & Multiplier LR & Float & $[10^{-3}, 1.0]$ & Uniform \\
\bottomrule
\end{tabular}
\end{table}

\section{Supplementary numerical illustrations}
In this section, several results regarding our numerical experiments are depicted.

\Cref{fig:scattter-eodds-pr} displays the training configurations for OT-FairBoost and competitor models, plotting PR AUC on the x-axis against the Equalized Odds gap ($\Delta\text{EOdds}$) on the y-axis. While FairGBM achieves strong predictive accuracy in certain configurations, its performance is widely scattered across regions that yield no systematic gain in either fairness or predictive power. In contrast, our proposed methodology consistently establishes the superior accuracy–fairness trade-off.

Similarly to \Cref{fig:scattter-eodds-pr}, \Cref{fig:scattter-multigroup-dponevsrest} illustrates the performance of OT-FairBoost in binary classification under multi-group sensitive attributes optimized for Demographic Parity ($\Delta\text{DP}$). Our method consistently achieves the optimal trade-off, with no competing baseline demonstrating comparable Pareto efficiency.

Finally, \Cref{fig:scattter-regression-w2-mae} presents the empirical results for continuous target settings (regression) with a binary sensitive attribute under DP constraints, evaluated using the Wasserstein-2 distance ($W_2$) and Mean Absolute Error ($\text{MAE}$) as the fairness and predictive metrics, respectively. In this regime, OT-FairBoost attains low $W_2$ parity gaps while maintaining a lower MAE than post-processing alternatives such as W2Post, thereby delivering a superior overall accuracy–fairness trade-off.
\begin{figure*}[htbp]
    \centering
    \begin{subfigure}[b]{0.32\textwidth}
        \centering
        \includegraphics[width=\linewidth]{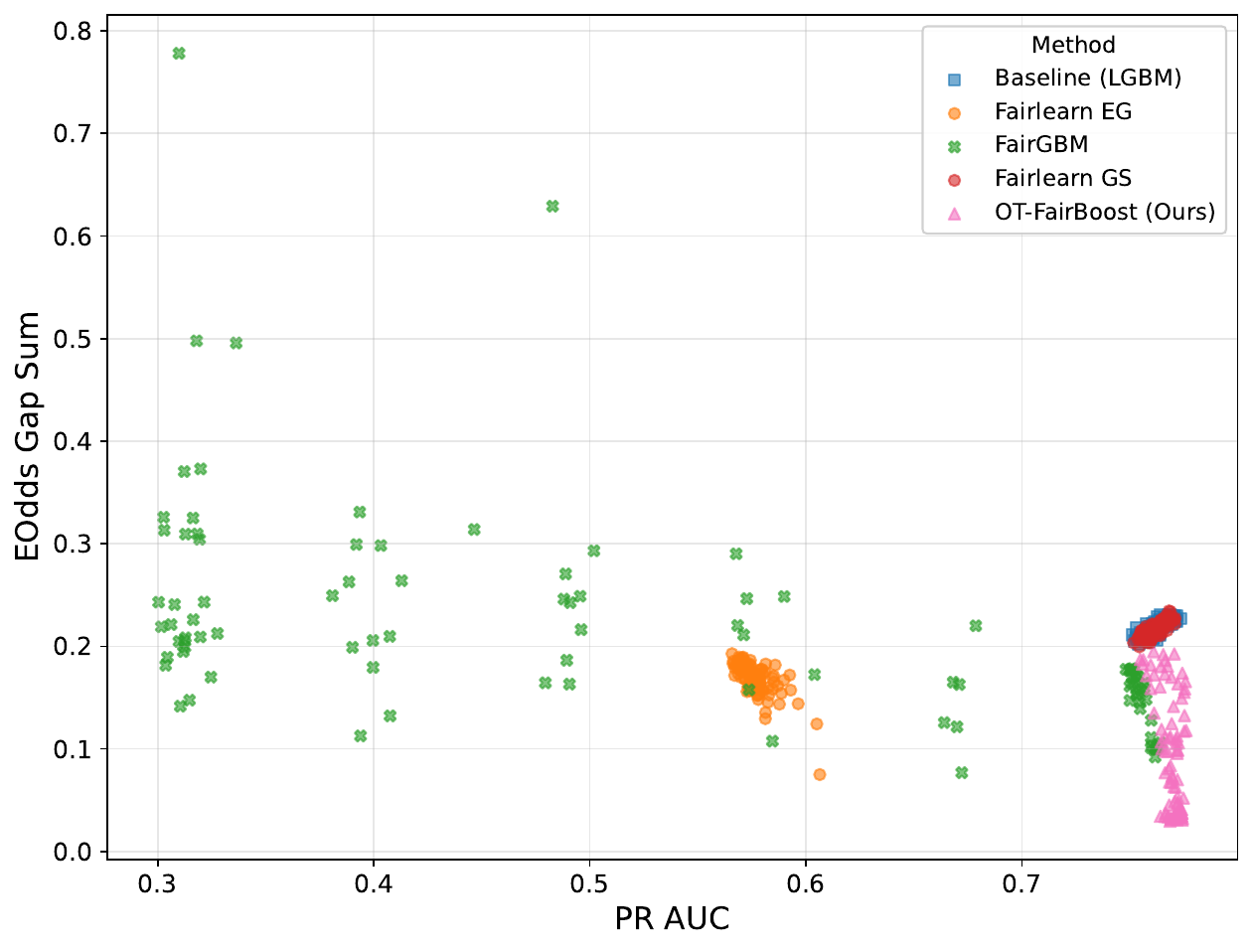}
        \caption{PR AUC vs. Equalized Odds gap for binary classification task.}
        \label{fig:scattter-eodds-pr}
    \end{subfigure}
    \hfill
    \begin{subfigure}[b]{0.32\textwidth}
        \centering
        \includegraphics[width=\linewidth]{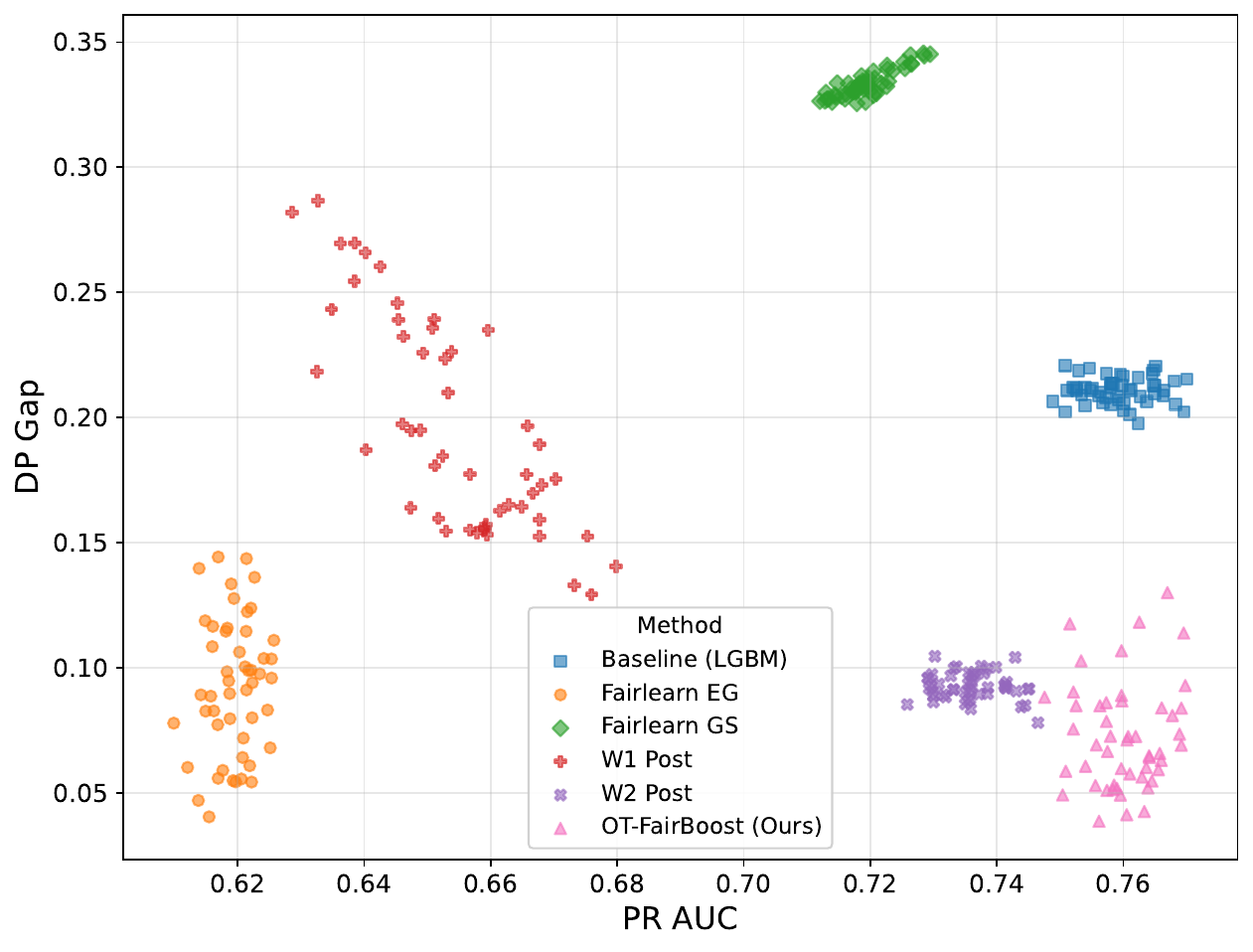}
        \caption{PR AUC vs. DP gap with multi-group sensitive attributes.}
        \label{fig:scattter-multigroup-dponevsrest}
    \end{subfigure}
    \hfill
    \begin{subfigure}[b]{0.32\textwidth}
        \centering
        \includegraphics[width=\linewidth]{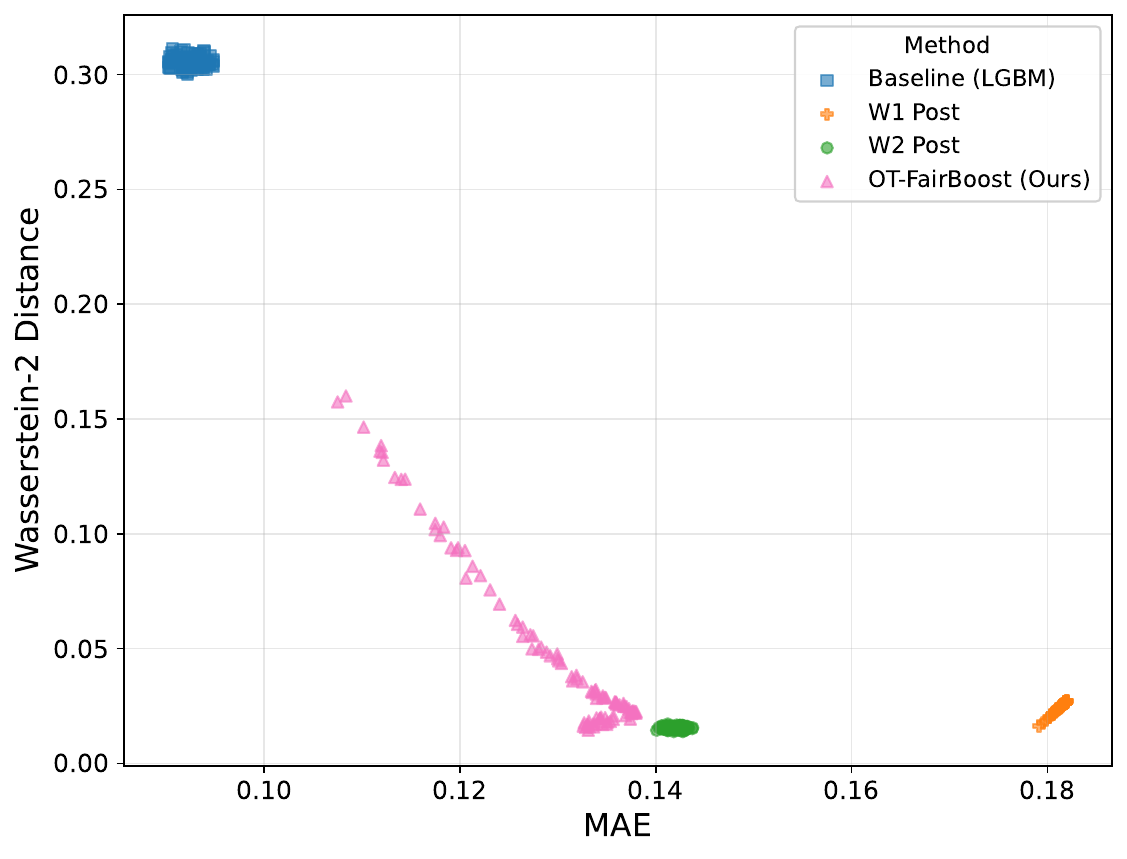}
        \caption{MAE vs. $\mathcal{W}_2$ for regression with binary sensitive attributes.}
        \label{fig:scattter-regression-w2-mae}
    \end{subfigure}
    
    \caption{Performance comparisons across different classification and regression tasks.}
    \label{fig:scatter-plots-combined}
\end{figure*}

%\clearpage
\section{Proof of Proposition 1}

\begin{proof}{}
\label{proof:proposition:main-derivatives}
The right-hand partial derivatives is defined as
$$
\frac{\partial^+ \mathcal{W}_2^2(\mu^{\mathcal{Z}^0}, \mu^{\mathcal{Z}^1})}{\partial {\hat z^0_i}}  = \lim_{\substack{\varepsilon \to 0 \\ \varepsilon>0}}
\frac{\mathcal{W}_2^2(\mu^{\mathcal{Z}_\varepsilon^0}, \mu^{\mathcal{Z}^1}) -
\mathcal{W}_2^2(\mu^{\mathcal{Z}^0}, \mu^{\mathcal{Z}^1})}{\varepsilon}
$$

and similarly for the hessian
$$
\frac{\partial^{2+} \mathcal{W}_2^2(\mu^{\mathcal{Z}^0}, \mu^{\mathcal{Z}^1})}{\partial {(\hat z^0_i})^2}=
\lim_{\substack{\varepsilon \to 0 \\ \varepsilon>0}} \frac{\mathcal{W}_2^2(\mu^{\mathcal{Z}_{2\varepsilon}^0}, \mu^{\mathcal{Z}^1}) - 2
\mathcal{W}_2^2(\mu^{\mathcal{Z}^0_{\varepsilon}}, \mu^{\mathcal{Z}^1}) + \mathcal{W}_2^2(\mu^{\mathcal{Z}^0}, \mu^{\mathcal{Z}^1})}{\varepsilon^2}
$$
%where we write $\mathcal{U}_\varepsilon^0$  the set of distinct values from $\mathcal{Y}_\varepsilon^0$, with the corresponding probabilities.
where $\frac{\partial^+ f}{\partial x}$ denotes the right-hand partial derivatives of function $f$ with respect to $x$.

In the following, we denote $\gamma^*$  the coupling matrix between $\mu^{\mathcal{Z}_\varepsilon^0}$ and $ \mu^{\mathcal{Z}^1}$, as defined in \Cref{eq:w2_unique}. It is unique for sufficiently small $\varepsilon > 0$ that thus keeps the same order between values in $\mathcal{U}_\varepsilon^0$ and their respective weights, i.e. $\varepsilon < \min_{i \neq j} |\hat z^0_i - \hat z^0_j|$. We will write the perturbed value $u^0_{d(i)} = \hat z^0_i + \varepsilon$, where $d$ is the mapping between $\mathcal{Z}^0_\varepsilon$ and $\mathcal{U}^0_\varepsilon$. 

In the case where $\hat z_i^0$ is unique, $|\mathcal{U}^0|=|\mathcal{U}^0_\varepsilon|$, $d$ is also a valid mapping for $\mathcal{Z}^0$ and $\mathcal{U}^0$, and we can isolate the term $j=d(i)$ in the sum of \Cref{eq:w_2_discrete}:
\begin{align}
\mathcal{W}_2^2(\mu^{\mathcal{Z}_\varepsilon^0}, \mu^{\mathcal{Z}^1}) \nonumber
& = \sum_{\substack{j=1 \\ j\neq d(i)}}^{|\mathcal{U}_\varepsilon^0|} \sum_{k=1}^{|\mathcal{U}^1|} \gamma^*_{j,k} ( u^0_j - u^1_k )^2 + \sum_{k=1}^{|\mathcal{U}^1|} \gamma^*_{d(i),k} ( \hat z^0_{i} + \varepsilon - u^1_k)^2
\nonumber
\\
& = \mathcal{W}^2_2(\mu^{\mathcal{Z}^0}, \mu^{\mathcal{Z}^1})
 + \sum_{k=1}^{|\mathcal{U}^1|} \gamma^*_{d(i),k} \big(( \hat z^0_i + \varepsilon - u^1_k )^2 - ( \hat z^0_i - u^1_k )^2\big)
\nonumber
\\
& = 
\mathcal{W}^2_2(\mu^{\mathcal{Z}^0}, \mu^{\mathcal{Z}^1})
+ \sum_{k=1}^{|\mathcal{U}^1|} \gamma^*_{d(i),k} \Big(  2 \varepsilon\big(\hat z^0_i - u^1_k\big) +\varepsilon^2 \Big)
\nonumber
\\
&=
\mathcal{W}^2_2(\mu^{\mathcal{Z}^0}, \mu^{\mathcal{Z}^1}) \nonumber
 + p^0\big(u_{d(i)}^0\big) \left(2\varepsilon\left(\hat z^0_i - T_{\mathcal{U}^0\to \mathcal{U}^1}\big(\hat z^0_i\big)\right) +\varepsilon^2 \right)
\nonumber
\end{align}
thus proving the gradient and hessian expressions, noting $p^0\big(u_{d(i)}^0\big) = \frac{1}{n^0}$. Please note that in this case, the derivative holds for any $\varepsilon$, and $\mathcal{W}^2_2(\mu^{\mathcal{Z}^0}, \mu^{\mathcal{Z}^1})$ is $\mathcal{C}^\infty$.

If $\hat z^0_i$ is not unique in $\mathcal{Z}^0$, then $|\mathcal{U}^0_\varepsilon|=|\mathcal{U}^0|+1$, and we have $u^0_{d(i)-1}=\hat z^0_i$ and $u^0_{d(i)}=\hat z^0_i + \varepsilon$.
If we isolate the perturbed index $d(i)$, we have:

\begin{align}
\mathcal{W}_2^2(\mu^{\mathcal{Z}_\varepsilon^0}, \mu^{\mathcal{Z}^1}) \nonumber =& \sum_{\substack{j=1 \\ j\neq d(i)}}^{|\mathcal{U}_\varepsilon^0|} \sum_{k=1}^{|\mathcal{U}^1|} \gamma^*_{j,k} ( u^0_j - u^1_k )^2 + \sum_{k=1}^{|\mathcal{U}^1|} \gamma^*_{d(i),k} ( \hat z^0_{i} - u^1_k)^2 \\
 &+ \sum_{k=1}^{|\mathcal{U}^1|} \gamma^*_{d(i),k} \big(( \hat z^0_{i} + \varepsilon - u^1_k )^2 - ( \hat z^0_{i} - u^1_k )^2\big)
\end{align}

The first line equals $\mathcal{W}_2^2(\mu^{\mathcal{Y}^0}, \mu^{\mathcal{Y}^0})$. There are indeed two terms in the sum with the same $u^0_j$ value (equals $\hat z^0_i$), that we can thus add together as
$$
\sum_{j=1}^{|\mathcal{U}^0|} \sum_{k=1}^{|\mathcal{U}^1|} \gamma_{j,k} ( u^0_j - u^1_k )^2
$$
where 
$$
\gamma_{j,k} = \begin{cases} 
  \gamma^*_{j,k} & \text{if } j < d(i)-1 \\
  \gamma^*_{d(i)-1,k} + \gamma^*_{d(i),k}   & \text{if } j = d(i) \\
  \gamma^*_{j+1,k} & \text{if } j \ge d(i) 
\end{cases}
$$
and is the solution for $\mathcal{W}_2^2(\mu^{\mathcal{Z}^0}, \mu^{\mathcal{Z}^0})$, since it corresponds to the correct iterative construction of $\gamma$, with allocations merged for $d(i)-1$ and $d(i)$.

We can then follow the same derivation as above in the case of unique $\hat z^0_i$ value, and obtain the same gradient and hessian expressions.
In this case, we need $\varepsilon>0$ in order to keep a valid ranking among values so that $\gamma$ is a solution for $\mathcal{W}_2^2(\mu^{\mathcal{Z}^0}, \mu^{\mathcal{Z}^0})$.

Note that we can similarly calculate the left-hand partial derivative
$$
\frac{\partial^- \mathcal{W}_2^2(\mu^{\mathcal{Z}^0}, \mu^{\mathcal{Z}^1})}{\partial {\hat z^0_i}}
= \frac{2}{N^0} \big(\hat z^0_i - T^<_{\mathcal{Z}^0 \to \mathcal{Z}^1} (\hat z^0_i)\big)
$$
where the mapping function $T^<_{\mathcal{Y}^0 \to \mathcal{Y}^1}(y^0_i)$ is
$$
T^<_{\mathcal{Z}^0 \to \mathcal{Z}^1}(z^0_i) = \lim_{\substack{\varepsilon \to 0 \\ \varepsilon < 0}} \; T_{\mathcal{Z}^0_{-\varepsilon} \to \mathcal{Z}^1} (\hat z^0_i + \varepsilon)
$$
where $\varepsilon$ is negative so that the underlying coupling matrix is for $\hat z_i^0$ coming from the left-hand side of its duplicates.

The left-hand hessian has the same expression as the right-hand hessian.

One notes that $
\frac{\partial^- \mathcal{W}_2^2(\mu^{\mathcal{Z}^0}, \mu^{\mathcal{Z}^1})}{\partial {\hat z^0_i}} \neq \frac{\partial^+ \mathcal{W}_2^2(\mu^{\mathcal{Z}^0}, \mu^{\mathcal{Z}^1})}{\partial {\hat z^0_i}}$, so that on duplicate values, the Wasserstein-2 is not $\mathcal{C}^1$. The difference between left and right-hand derivatives is however tending to zero when the number of samples $|\mathcal{Z}^0|$ goes to infinity, confirming our approach for numerical estimation based on cumulative distribution functions.

The derivative is equivalent for a sample $\hat z^1_j\in \mathcal{Z}^1$, with the transport map $T_{\mathcal{Z}^1 \to \mathcal{Z}^0}$.

%The calculation for the second order equivalently involve considering the values set $\mathcal{Y}_{2\varepsilon}^0$ with a perturbation $y^0_i + 2 \varepsilon$, and compute the corresponding Wasserstein distance 
%$\mathcal{W}_2^2(\mu^{\mathcal{Y}_{2\varepsilon}^0}, \mu^{\mathcal{Y}_{\varepsilon}^1})$. For $\varepsilon \to 0$ the corresponding allocation solution $\gamma^*$ is the same as above. Second-order forward formula:

%We can separately consider the first terms
%\begin{align}
%&\frac{\mathcal{W}_2^2(\mu^{\mathcal{Y}_\varepsilon^0}, \mu^{\mathcal{Y}^1}) -
%\mathcal{W}_2^2(\mu^{\mathcal{Y}^0}, \mu^{\mathcal{Y}^1})}{\varepsilon^2} \nonumber \\
%&= \frac{1}{\varepsilon N^0} (2 y^0_i - 2 \sum_{k=1}^{|\mathcal{U}^1|} N^0\gamma^*_{i,k} u^1_k) + \frac{1}{N^0}
%\end{align}

\end{proof}

%\clearpage
\section{Samplewise Derivatives of Continuous Wasserstein Distance}

We denote by $\mu_0$ and $\mu_1$ the output distributions of $\hat{z}$ for observations in the groups $S = 0$ and $S = 1$, respectively, and denote by $h_0$ and $h_1$ their densities. Their corresponding cumulative distribution functions are $H_0$ and $H_1$. The Wasserstein-2 distance between the two conditional distributions is defined as

\begin{equation}
    \label{eq:w2_dist}
    \mathcal{W}^2_2(\mu_0,\mu_1) = \int_{0}^1 \Big(H^{-1}_0(\tau) - H^{-1}_1(\tau)\Big)^2d\tau, 
\end{equation}

where $H^{-1}_s$ is the inverse of the cumulative distribution function $H_s$. In the following, we follow \cite{risser2022tackling} that give a meaning to the pseudo-derivative of $\mathcal{W}^2_2(\mu_0,\mu_1)$ with respect to a specific observation $\hat{z_i}$. In the following, we omit $\hat{\cdot}$ for readability.

\subsection{General Gâteaux Differentiability Model}
The transport cost \Cref{eq:w2_dist} is Gâteaux differentiable in the direction $(\alpha, \beta)$, with derivative $\mathcal{DW}^2_2(\mu_0, \mu_1)(\alpha, \beta)$ if the limit

\begin{align}
    \label{eq:dw2_def}
    &\mathcal{DW}^2_2(\mu_0, \mu_1)(\alpha, \beta) := \lim_{t \rightarrow 0} \frac{\mathcal{W}^2_2(\mu_0 + t\alpha, \mu_1 + t\beta) - \mathcal{W}^2_2(\mu_0, \mu_1)}{t}
\end{align}

exits and is finite.

By only considering the $\alpha$ term, i.e. we perturb only the distribution $\mu_0$, \Cref{eq:dw2_def} can the be written as

\begin{align}
     &\mathcal{DW}^2_2(\mu_0, \mu_1)(\alpha, 0) :=  \lim_{t \rightarrow 0} \int_{0}^{1} \frac{(H^{-1}_{\mu_0 + t\alpha}(\tau) - H^{-1}_{\mu_1}(\tau))^2 - (H^{-1}_{\mu_0}(\tau) - H^{-1}_{\mu_1}(\tau))^2}{t}d\tau
\end{align}

By factorizing this equation, we can show that it is equal to

\begin{align}
    \label{eq:dw2-factorized}
    &\mathcal{DW}^2_2(\mu_0, \mu_1)(\alpha, 0) := \lim_{t \rightarrow 0} \int_{0}^{1} \frac{(H^{-1}_{\mu_0 + t\alpha} + H^{-1}_{\mu_0} - 2H^{-1}_{\mu_1})(H^{-1}_{\mu_0 + t\alpha} - H^{-1}_{\mu_0})}{t}d\tau,
\end{align}

where we omitted $\tau$ for readability. Taking the limit, \Cref{eq:dw2-factorized} can be rewritten as 

\begin{align}
    \mathcal{DW}^2_2(\mu_0, \mu_1)(\alpha, 0) = 2 \int_{0}^{1}(H^{-1}_{\mu_0} - H^{-1}_{\mu_1})\frac{d}{dt}H^{-1}_{\mu_0 + t\alpha}\rvert_{t=0}d\tau.
\end{align}

We can show that 
\begin{equation}
    \label{eq:dh}
    \frac{d}{dt}H^{-1}_{\mu_0 + t\alpha} = \frac{-\alpha(-\infty, H^{-1}_{\mu_0 + t\alpha}(\tau))}{\mu_0(H^{-1}_{\mu_0 + t\alpha}(\tau)) + t\alpha(H^{-1}_{\mu_0 + t\alpha}(\tau))},
\end{equation}

where $\alpha(-\infty, x) = \int_{-\infty}^{x}d\alpha$.

Therefore, 

\begin{align}
    \label{eq:dw2_final}
    \mathcal{DW}^2_2(\mu_0, \mu_1)(\alpha, 0) =  2 \int_{0}^{1}(H^{-1}_{\mu_0} - H^{-1}_{\mu_1}) \frac{-\alpha(-\infty, H^{-1}_{\mu_0}(\tau))}{\mu_0(H^{-1}_{\mu_0}(\tau))}d\tau.
\end{align}

\subsection{Influence of the Perturbation of an Output $z_i$ on the Wasserstein-2 Distance}

In order to model the impact of a local perturbation around $z_i$ on $\mathcal{W}^2_2(\mu_0,\mu_1)$, with $s_i = 0$, we set $\alpha_{\varepsilon}^i = \Delta\mathds{1}_{[z_i, z_i + \varepsilon]} - \Delta\mathds{1}_{[z_i - \varepsilon, z_i[}$, with $\varepsilon > 0$ and $\Delta>0$.

We set $z = H_{\mu_0}^{-1}(\tau)$ in \Cref{eq:dw2_final}. This means $\tau = H_{\mu_0}(z)$, $d\tau = \mu_0(z)dz$ and

\begin{align}
    \label{eq:dw2-var-change}
    &\mathcal{DW}^2_2(\mu_0, \mu_1)(\alpha_{\varepsilon}^i, 0) = -2 \int_{z_i - \varepsilon}^{z_i + \varepsilon}(z - H^{-1}_{\mu_1}(H_{\mu_0}(z)) \alpha_{\varepsilon}^i(-\infty, z)dz.
\end{align}

For $z \in [z_i - \varepsilon, z_i + \varepsilon]$, we have

\begin{align}
    \alpha^i_{\varepsilon}(H_{\mu_0}^{-1}(\tau)) &= \alpha^i_{\varepsilon}(z) = \Delta\sgn(z - z_i),
    \\
    \alpha^i_{\varepsilon}(-\infty, H_{\mu_0}^{-1}(\tau)) &= \alpha^i_{\varepsilon}(-\infty, z) = \Delta(|z - z_i| - \varepsilon).
\end{align}

We can rewrite \Cref{eq:dw2-var-change} in the following way

\begin{align}
    \label{eq:dw2-two-terms}
    &\mathcal{DW}^2_2(\mu_0, \mu_1)(\alpha_{\varepsilon}^i, 0)  \nonumber =2\varepsilon\Delta\int_{z_i - \varepsilon}^{z_i + \varepsilon}(z - H^{-1}_{\mu_1}(H_{\mu_0}(z))dz - 2\Delta\int_{z_i - \varepsilon}^{z_i + \varepsilon}(z - H^{-1}_{\mu_1}(H_{\mu_0}(z))|z - z_i|dz.
\end{align}

% We can show that the first term is equal to $2\varepsilon(z_i - H^{-1}_{\mu_1}(H_{\mu_0}(z_i)) + $.

To calculate the integrals, we set $A(z) := z - H^{-1}_{\mu_1}(H_{\mu_0}(z))$ . The Taylor's expansion around $z_i$ for $A(z)$ is

\begin{equation}
    A(z) = A(z_i) + A'(z_i)(z - z_i) + \mathcal{O}(\varepsilon^2),
\end{equation}

where $|z - z_i| \leq \varepsilon$. (\textit{NB}: $A(z)$ should be in $\mathcal{C}^1$).

Splitting the second term in \Cref{eq:dw2-two-terms} into two integrals $\int_{z_i - \varepsilon}^{z_i}$ and $\int_{z_i}^{z_i + \varepsilon}$ and substituting Taylor's expansion, we obtain

\begin{align}
    &2\varepsilon\Delta A(z_i)\int_{z_i - \varepsilon}^{z_i + \varepsilon}dz + 2\varepsilon\Delta A'(z_i)\int_{z_i - \varepsilon}^{z_i + \varepsilon}(z - z_i)dz  \nonumber\\ 
    &+ 2\Delta A(z_i)\int_{z_i - \varepsilon}^{z_i}(z - z_i)dz +  2\Delta A'(z_i)\int_{z_i - \varepsilon}^{z_i}(z - z_i)^2dz  \nonumber
    \\ 
    &- 2\Delta A(z_i)\int_{z_i}^{z_i + \varepsilon}(z - z_i)dz - 2\Delta A'(z_i)\int_{z_i}^{z_i + \varepsilon}(z - z_i)^2dz + \mathcal{O}(\varepsilon^3) \nonumber
    \\ 
    &= 4\varepsilon^2\Delta A(z_i) - \varepsilon^2\Delta A(z_i) + \frac{2\varepsilon^3}{3}\Delta A'(z_i) - \varepsilon^2\Delta A(z_i) - \frac{2\varepsilon^3}{3}\Delta A'(z_i) + \Delta\mathcal{O}(\varepsilon^4).
\end{align}

Finally,

\begin{align}
    &\mathcal{DW}^2_2(\mu_0, \mu_1)(\alpha_{\varepsilon}^i, 0) = 2\varepsilon^2\Delta(z_i - H^{-1}_{\mu_1}(H_{\mu_0}(z_i)) +  \Delta\mathcal{O}(\varepsilon^4).
\end{align}

By setting in the perturbation $\alpha^i_{\varepsilon}$, $\Delta = \frac{1}{\varepsilon^2 n^0}$ and taking $\varepsilon$ to 0, we obtain the same expression of the gradient as in \Cref{proposition:main-derivatives}

\begin{align}
    &\mathcal{DW}^2_2(\mu_0, \mu_1)(\alpha_{\varepsilon}^i, 0)  \xrightarrow[\varepsilon \to 0]{} \frac{2}{n^0}(z_i - H^{-1}_{\mu_1}(H_{\mu_0}(z_i)).
\end{align}

\subsection{Extension to Second Derivatives}
The transport cost \Cref{eq:w2_dist} has the second Gâteaux derivative $\mathcal{DW}^2_2(\mu_0, \mu_1)(\alpha, \beta)$ in the direction $(\alpha, \beta)$, if the limit

\begin{align}
    \label{eq:d2w2_def}
    &\mathcal{D}^2\mathcal{W}^2_2(\mu_0, \mu_1)(\alpha, \beta) := \lim_{t \rightarrow 0} \frac{1}{t^2} \Bigg[\mathcal{W}^2_2(\mu_0 + t\alpha, \mu_1 + t\beta) - 2\mathcal{W}^2_2(\mu_0, \mu_1)  + \mathcal{W}^2_2(\mu_0 - t\alpha, \mu_1 - t\beta)\Bigg]
\end{align}
exits and is finite.

By only considering the $\alpha$ term, i.e. we perturb only the distribution $\mu_0$, \Cref{eq:d2w2_def} can the be written as

\begin{align}
     \mathcal{D}^2\mathcal{W}^2_2(\mu_0, \mu_1)(\alpha, 0) := \lim_{t \rightarrow 0} \frac{1}{t^2} \int_{0}^{1} \Bigg[&(H^{-1}_{\mu_0 + t\alpha}(\tau) - H^{-1}_{\mu_1}(\tau))^2  - 2(H^{-1}_{\mu_0}(\tau) - H^{-1}_{\mu_1}(\tau))^2 
     \nonumber
     \\ &
     +(H^{-1}_{\mu_0 - t\alpha}(\tau) - H^{-1}_{\mu_1}(\tau))^2\Bigg]d\tau
\end{align}

By factorizing this equation, we can show that it is equal to

\begin{align}
    \label{eq:d2w2-factorized}
    \mathcal{D}^2\mathcal{W}^2_2(\mu_0, \mu_1)(\alpha, 0) :=
    \lim_{t \rightarrow 0} \frac{1}{t^2} \int_{0}^{1}& \Bigg[ (H^{-1}_{\mu_0 + t\alpha} - H^{-1}_{\mu_0})^2 + (H^{-1}_{\mu_0 - t\alpha} - H^{-1}_{\mu_0})^2 + \nonumber
    \\ 
    &+ 2(H^{-1}_{\mu_0} - H^{-1}_{\mu_1})(H^{-1}_{\mu_0 + t\alpha} - 2H^{-1}_{\mu_0} + H^{-1}_{\mu_0 - t\alpha}) \Bigg]d\tau,
\end{align}

where we omitted $\tau$ for readability. Taking the limit, \Cref{eq:d2w2-factorized} can be rewritten as

\begin{align}
    \label{eq:d2w2_final}
    &\mathcal{D}^2\mathcal{W}^2_2(\mu_0, \mu_1)(\alpha, 0) =  2 \int_{0}^{1} \Bigg[\Bigg(\frac{d}{dt}H^{-1}_{\mu_0 + t\alpha}\rvert_{t=0}\Bigg)^2 + (H^{-1}_{\mu_0} - H^{-1}_{\mu_1})\frac{d^2}{dt^2}H^{-1}_{\mu_0 + t\alpha}\rvert_{t=0}\Bigg]d\tau.
\end{align}

We can show that
\begin{align}
    \label{eq:d2h}
    \frac{d^2}{dt^2}H^{-1}_{\mu_0 + t\alpha} =& \frac{1}{(\mu_0(H^{-1}_{\mu_0 + t\alpha}) + t\alpha(H^{-1}_{\mu_0 + t\alpha}))^2} \cdot \nonumber \\ 
    &\Bigg[-\alpha(H^{-1}_{\mu_0 + t\alpha})\frac{d}{dt}H^{-1}_{\mu_0 + t\alpha}\Big(\mu_0(H^{-1}_{\mu_0 + t\alpha}) + t\alpha(H^{-1}_{\mu_0 + t\alpha})\Big) + \nonumber 
    \\
    & + \alpha(-\infty, H^{-1}_{\mu_0 + t\alpha})\Big(\mu^{\prime}_0(H^{-1}_{\mu_0 + t\alpha})\frac{d}{dt}H^{-1}_{\mu_0 + t\alpha} + \alpha(H^{-1}_{\mu_0 + t\alpha}) + \nonumber \\
    & + t\alpha^\prime(H^{-1}_{\mu_0 + t\alpha})\frac{d}{dt}H^{-1}_{\mu_0 + t\alpha}\Big)\Bigg].
\end{align}

At $t=0$, \Cref{eq:dh} and \Cref{eq:d2h} reduce to
\begin{align}
    \frac{d}{dt}H^{-1}_{\mu_0+t\alpha}\Big|_{t=0}
  = \frac{-\alpha(-\infty,H^{-1}_{\mu_0})}{\mu_0(H^{-1}_{\mu_0})}, 
\end{align}

\begin{align}
    &\frac{d^2}{dt^2}H^{-1}_{\mu_0+t\alpha}\Big|_{t=0}
    = \frac{2\,\alpha(H^{-1}_{\mu_0})\,\alpha(-\infty,H^{-1}_{\mu_0})}{\mu_0(H^{-1}_{\mu_0})^2}
   -  \frac{\mu_0'(H^{-1}_{\mu_0})\,\alpha(-\infty,H^{-1}_{\mu_0})^2}{\mu_0(H^{-1}_{\mu_0})^3},
\end{align}

We set $z = H_{\mu_0}^{-1}(\tau)$ in \Cref{eq:d2w2_final}. This means $\tau = H_{\mu_0}(z)$, $d\tau = \mu_0(z)dz$ and

\begin{align}
    \label{eq:d2w2-var-change}
    &\mathcal{D}^2\mathcal{W}^2_2(\mu_0, \mu_1)(\alpha, 0) = 2\int_0^1\Bigg[
   \frac{\alpha(-\infty,z)^2}{\mu_0(z)}
    + \frac{2\,A(z)\,\alpha(z)\,\alpha(-\infty,z)}{\mu_0(z)} - \frac{A(z)\,\mu_0'(z)\,\alpha(-\infty,z)^2}{\mu_0(z)^2}
   \Bigg]dz.
\end{align}

We set $\alpha_{\varepsilon}^i = \Delta\mathds{1}_{[z_i, z_i + \varepsilon]} - \Delta\mathds{1}_{[z_i - \varepsilon, z_i[}$, with $\varepsilon > 0$ and $\Delta>0$, as in the previous calculations. Then,

\begin{align}
\label{eq:d2w2-three-terms}
\mathcal{D}^2\mathcal{W}_2^2(\mu_0,\mu_1)(\alpha^i_\varepsilon,0) =  2\Delta^2\!\int_{z_i-\varepsilon}^{z_i+\varepsilon}\Bigg[ &
   \frac{(|z - z_i|-\varepsilon)^2}{\mu_0(z)}  
   + \frac{2\,A(z)\,\mathrm{sgn}(z - z_i)\,(|z - z_i|-\varepsilon)}{\mu_0(z)} -\nonumber \\
&- \frac{A(z)\,\mu_0'(z)\,(|z - z_i|-\varepsilon)^2}{\mu_0(z)^2}
   \Bigg]dz.
\end{align}

To calculate the integrals, we set $B(z) := \frac{1}{\mu_0(z)}$, $C(z) := \frac{A(z)}{\mu_0(z)}$ and $D(z) := \frac{A(z)\mu^\prime_0(z)}{\mu_0^2(z)}$.

We note that

\begin{align}
    &\int_{z_i-\varepsilon}^{z_i+\varepsilon}(|z - z_i|-\varepsilon)^2\,dz = \frac{2\varepsilon^3}{3}, \\
    &\int_{z_i-\varepsilon}^{z_i+\varepsilon}(z - z_i)\,(|z - z_i|-\varepsilon)^2\,dz = 0, \\
    &\int_{z_i-\varepsilon}^{z_i+\varepsilon}\sgn(z - z_i)\,(|z - z_i|-\varepsilon)\,dz = 0, \\
    &\int_{z_i-\varepsilon}^{z_i+\varepsilon}|z - z_i|\,(|z - z_i|-\varepsilon)\,dz = -\frac{\varepsilon^3}{3},
\end{align}

Taking it into account and substituting Taylor's expansions of $B(z), C(z)$ and $D(z)$ in \Cref{eq:d2w2-three-terms}, we obtain

\begin{align}
    \mathcal{D}^2\mathcal{W}_2^2(\mu_0,\mu_1)(\alpha^i_\varepsilon,0) &= \frac{4\varepsilon^3\Delta^2}{3} \Big[B(z_i) - C^\prime(z_i) - D(z_i)\Big]  + \Delta^2\mathcal{O}(\varepsilon^5)  \nonumber \\
    & =\frac{4\varepsilon^3\Delta^2}{3}\frac{1-A'(z_i)}{\mu_0(z_i)} + \Delta^2\mathcal{O}(\varepsilon^5).
\end{align}

We can show that

\begin{align}
    \frac{1 - A'(z_i)}{\mu_0(z_i)} = \frac{1}{\mu_1 \big(H^{-1}_{\mu_1}(H_{\mu_0}(z_i))\big)}.
\end{align}

Hence,

\begin{align}
    &\mathcal{D}^2\mathcal{W}_2^2(\mu_0,\mu_1)(\alpha^i_\varepsilon,0) = \frac{4\Delta^2\varepsilon^3}
     {3\,\mu_1 \big(H^{-1}_{\mu_1}(H_{\mu_0}(z_i))\big)}
     + \Delta^2O(\varepsilon^5).
\end{align}

By setting in the perturbation $\alpha^i_{\varepsilon}$, $\Delta = \frac{1}{\varepsilon}$, we obtain

\begin{align}
    \mathcal{D}^2\mathcal{W}_2^2(\mu_0,\mu_1)(\alpha^i_\varepsilon,0) &= \varepsilon \frac{4}
     {3\,\mu_1 \big(H^{-1}_{\mu_1}(H_{\mu_0}(z_i))\big)}
     + o(\varepsilon^2) \xrightarrow[\varepsilon \to 0]{} 0.
\end{align}

%\section{Algorithm for Discrete Wasserstein computation}

%Blabla. Pseudo Code. Ranking between samples is key.

%\clearpage
\section{Non-Diagonal Hessian Terms}

For non-diagonal hessian terms, we consider the following partial derivative:

\begin{align*}
    & \frac{\partial^- \partial^+ \mathcal{W}_2^2(\mu^{\mathcal{Z}^0}, \mu^{\mathcal{Z}^1})}{\partial^- {\hat z^0_i} \partial^+ {\hat z^0_j}}  =
     \lim_{\substack{\varepsilon \to 0 \\ \varepsilon<0}}
\frac{\frac{\partial^+ \mathcal{W}_2^2(\mu^{\mathcal{Z}_\varepsilon^0}, \mu^{\mathcal{Z}^1})}{\partial \hat z_j^0} -
\frac{\partial^+ \mathcal{W}_2^2(\mu^{\mathcal{Z}^0}, \mu^{\mathcal{Z}^1})}{\partial \hat z_j^0}}{\varepsilon}
\end{align*}
where $\mathcal{Z}_\varepsilon^0$ is the set of values with $z_j^0$ is perturbed by $+\varepsilon$. Since $\varepsilon<0$, the order in $\mathcal{Z}_\varepsilon^0$ is the same as the one in the definition of the transport map $T_{\mathcal{Z}^0 \to \mathcal{Z}^1}(\hat z^0_j)$ (\Cref{eq:mapping}) involved in $\frac{\partial^+ \mathcal{W}_2^2(\mu^{\mathcal{Z}_\varepsilon^0}, \mu^{\mathcal{Z}^1})}{\partial \hat z_j^0}$ (obvious if $\hat z_j^0> \hat z_i^0$, and also holds if $\hat z_j^0 < \hat z_i^0$ for sufficiently small $\varepsilon$). Thus, the coupling matrix $\gamma$ is preserved and $T_{\mathcal{Z}^0 \to \mathcal{Z}^1}(\hat z^0_j)=T_{\mathcal{Z}_\varepsilon^0 \to \mathcal{Z}^1}(\hat z^0_j)$, leading to
$$
\frac{\partial^- \partial^+ \mathcal{W}_2^2(\mu^{\mathcal{Z}^0}, \mu^{\mathcal{Z}^1})}{\partial^- {\hat z^0_i} \partial^+ {\hat z^0_j}}=0
$$

Note that the same calculation holds for $
\frac{\partial^+ \partial^- \mathcal{W}_2^2(\mu^{\mathcal{Z}^0}, \mu^{\mathcal{Z}^1})}{\partial^+ {\hat z^0_i} \partial^- {\hat z^0_j}}=0$.

In the discrete formulation, $
\frac{{\partial^+}^2 \mathcal{W}_2^2(\mu^{\mathcal{Z}^0}, \mu^{\mathcal{Z}^1})}{\partial {\hat z^0_i} \partial {\hat z^0_j}}$ and $
\frac{{\partial^-}^2 \mathcal{W}_2^2(\mu^{\mathcal{Z}^0}, \mu^{\mathcal{Z}^1})}{\partial {\hat z^0_i} \partial {\hat z^0_j}}$
are not defined, since the limit does not exist. While theoretically a problem, in practice the approximation we make by considering $=0$ has limited impact. First, it is only for duplicate values, that are a minority in general practical cases. Second, as shown above with our continuous formalism derivations, this problem of duplicates in the discrete formulation disapears when studying the Wasserstein on continuous distributions. Moreover, this approximation means to neglect the ill-defined cases of non-diagonal hessian terms, which falls into the sample-wise derivatives that are the gradient boosting philosophy. Last, we have numerical illustrations confirming that the above approximation has limited impact in practice, that we describe in the following.

Figure \ref{fig:scattter-oracle-hessian-matrix} shows empirical oracle hessian matrix (on the first group $\mathcal{Z^0}$ samples), with a protocol similar to the end of Section Model, with the same one-sided derivatives as above.

The following section illustrate that when incorporated in the gradient boosting objective, our derivatives are indeed minimizing the Wasserstein-2 distance $\mathcal{W}_2^2(\mu^{\mathcal{Z}^0},\mu^{\mathcal{Z}^1})$.

\begin{figure}[h]
    \centering
    \includegraphics[width=0.5\linewidth]{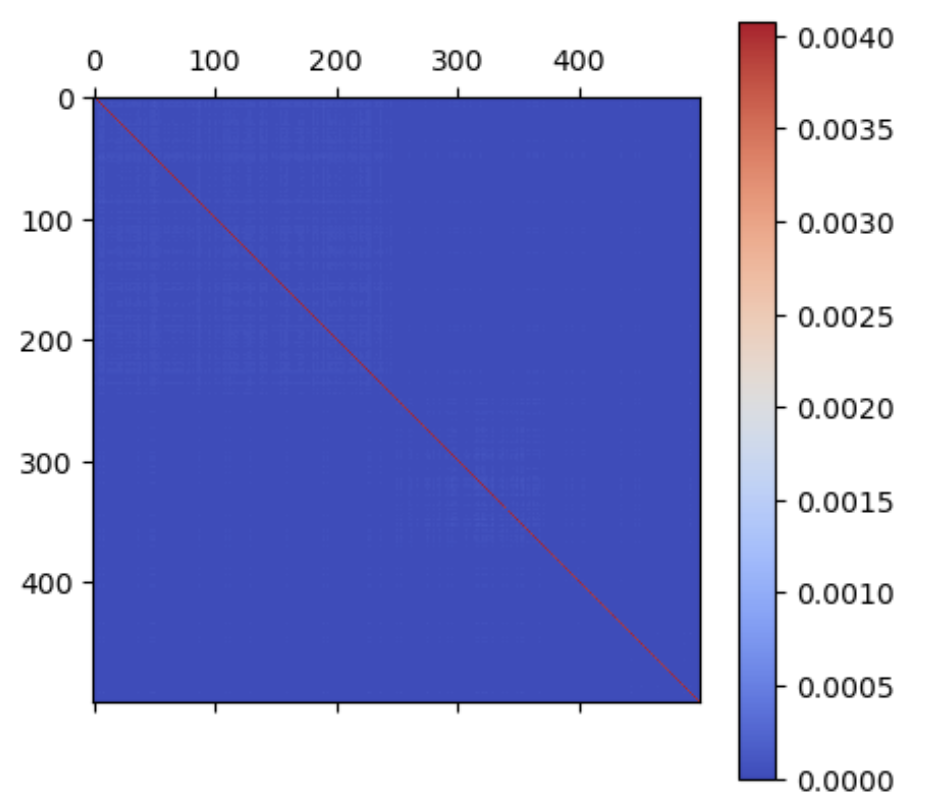}
    \caption{Hessian matrix from oracle estimates, on the first group $\mathcal{Z}^0$ samples ($n^0=500$).}
    \label{fig:scattter-oracle-hessian-matrix}
\end{figure}

\section{Empirical Convergence}

In this section, we consider fixed hyperparameters for OT-FairBoost optimizing DP for binary classification, except for $\lambda$ that varies. We illustrate on the train data how the empirical $\mathcal{W}_2^2(\mu^{\mathcal{Z}^0},\mu^{\mathcal{Z}^1})$ is evolving with respect to the boosting iterations $t$, with $\mathcal{Z}^0$ and $\mathcal{Z}^1$ updated at each iteration. Figure \ref{fig:cdfs-by-iterations-TN} shows their CDF at various iterations, showing two behaviors: the predictions $\hat z_i$ are more and more spread during the optimization, starting by a dirac on the average of ground truths $y_i$ (a well-known behavior for gradient descent algorithms); and the distributions $\mathcal{Z}^0$ and $\mathcal{Z}^1$ are more and more aligned when $\lambda$ increases, as our method aims to.
\begin{figure*}[h]
    \centering
    \includegraphics[width=1.0\linewidth]{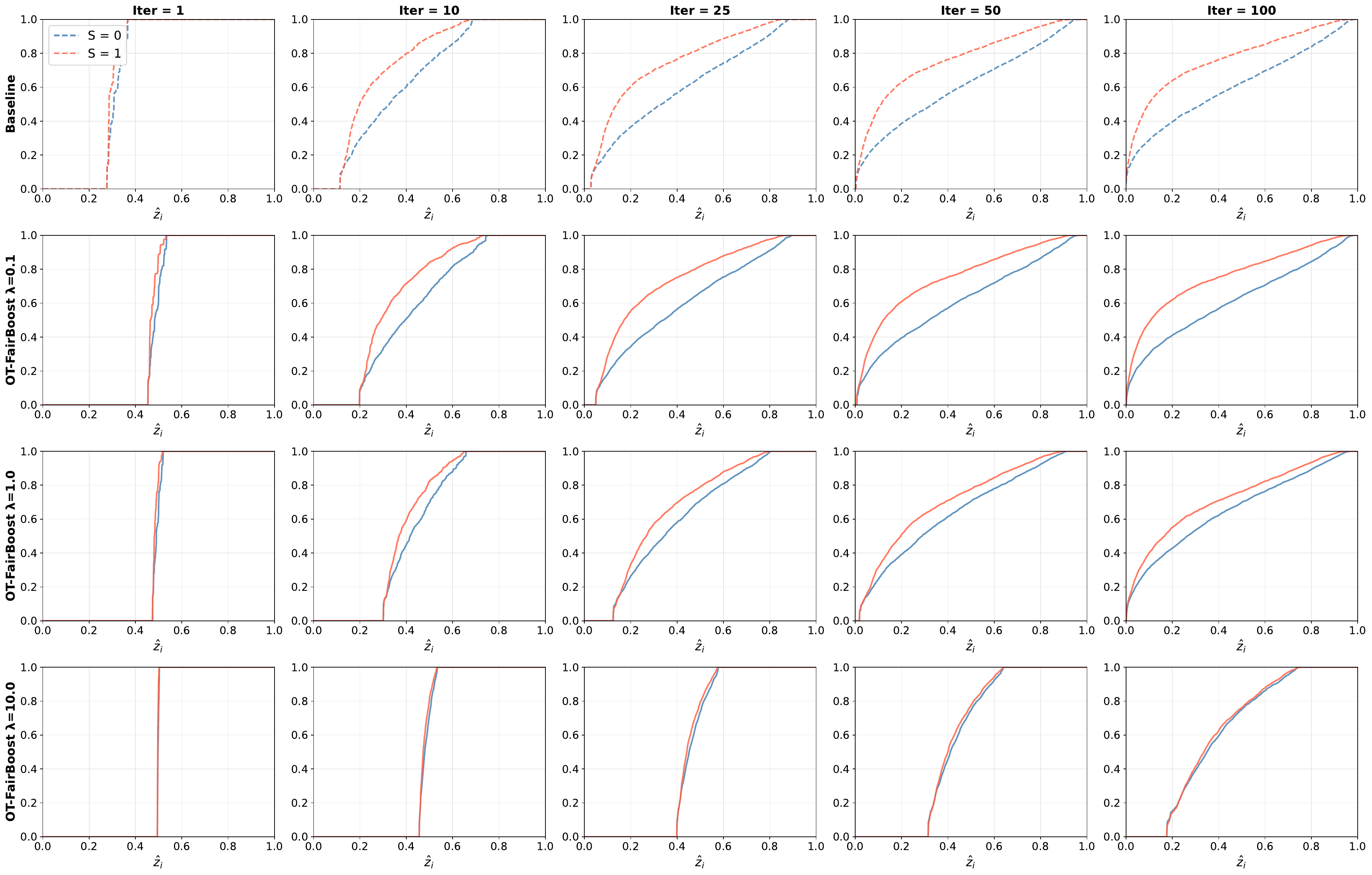}
    \caption{CDFs of $\mathcal{Z}^0$ and $\mathcal{Z}^1$ with respect to boosting iteration $t$: each line is for a $\lambda$ value ($\lambda = 0$, $0.1$, $1$ and $10$), and each row for an iteration value ($t=1$, $10$, $25$, $50$ and $100$).}
    \label{fig:cdfs-by-iterations-TN}
\end{figure*}

This optimization is a trade-off between the predictive performance and the fairness, where the first one is spreading more and more the predictions during the iterative gradient descent. This spreading makes the Wasserstein distance minimization harder to meet, as the differences $|\hat z_i^0 - \hat z_j^1|$ can be larger. To discard this effect, we thus consider a normalized Wasserstein distance:
\begin{align}
\bar{\mathcal{W}}_2(\mathcal{Z}^0, \mathcal{Z}^1) &= \frac{\mathcal{W}^2_2(\mathcal{Z}^0, \mathcal{Z}^1)}{\sigma_{\mathcal{Z}^0}^2 + \sigma_{\mathcal{Z}^1}^2}
\nonumber
=\frac{\mathcal{W}^2_2(\mathcal{Z}^0, \mathcal{Z}^1)}{\mathcal{W}^2_2(\mathcal{Z}^0, \delta_{<\mathcal{Z}^0>}) + \mathcal{W}^2_2(\mathcal{Z}^1, \delta_{<\mathcal{Z}^1>})}
\nonumber
\end{align}
where $\sigma_{\mathcal{Z}^0}$ and $<\mathcal{Z}^0>$ respectively denotes the standard deviation and average of points $\mathcal{Z}^0$. We indeed can interpret the  variance $\sigma_{\mathcal{Z}^0}^2$ as the square Wasserstein-2 distance to the dirac distribution on the mean value $<\mathcal{Z}^0>$, thus proving the homogeneity of this normalization. The normalized Wasserstein-2 distance is derived in \Cref{fig:w2-normalized-by_iterations-TN} over boosting iterations for usual LightGBM and OT-FairBoost. We observe that the normalized Wasserstein-2 decreases as the boosting iterations increase, with stronger minimization for higher $\lambda$ values.

\begin{figure*}[h]
    \centering
    \includegraphics[width=0.8\linewidth]{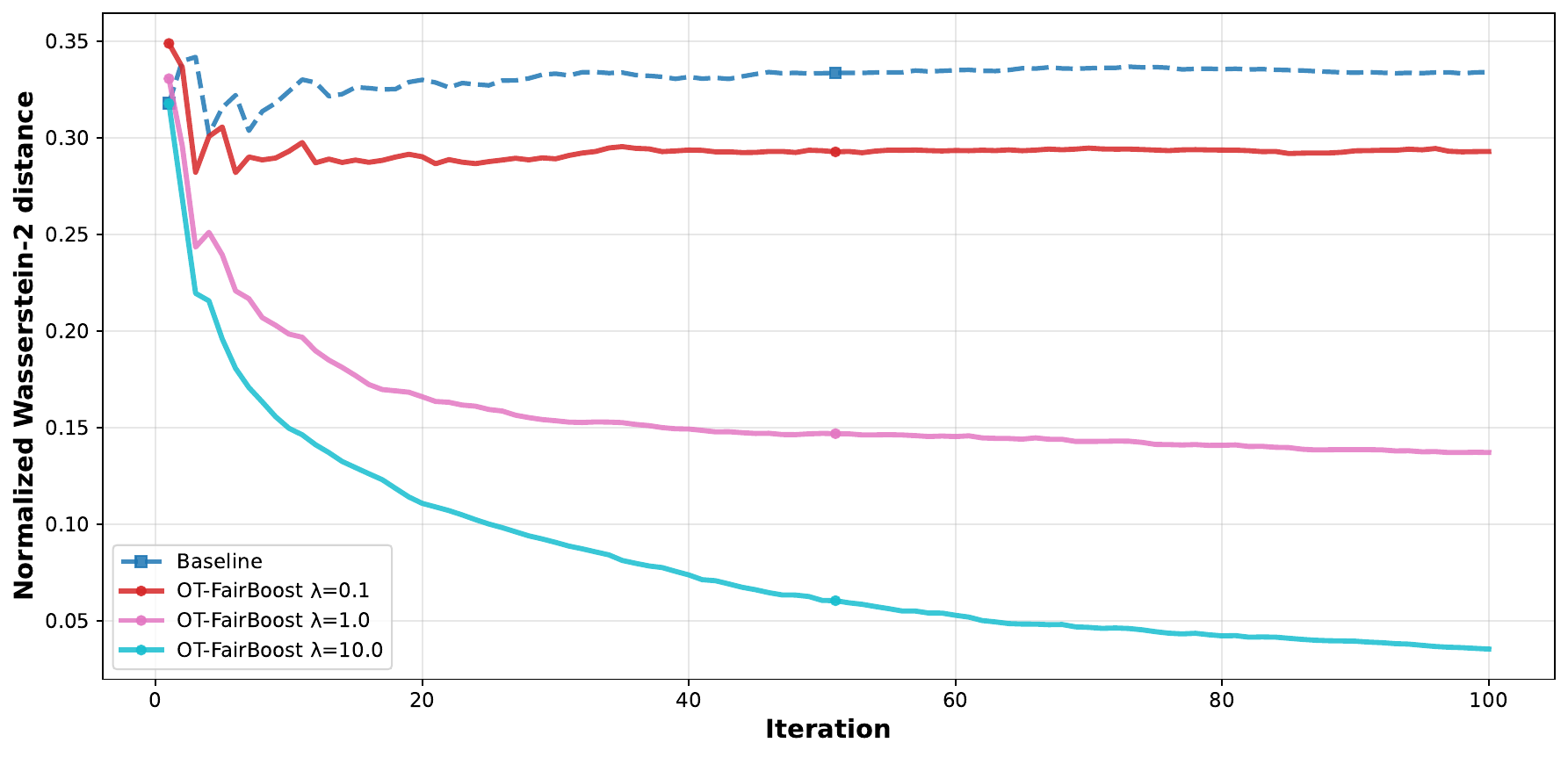}
    \caption{$\mathcal{W}_2$ normalized with respect to iterations.}
    \label{fig:w2-normalized-by_iterations-TN}
\end{figure*}

\end{document}